\documentclass[a4paper]{amsart}
\usepackage{amsmath,amsthm,amssymb,latexsym,epic,bbm,comment}
\usepackage{graphicx,enumerate,stmaryrd,color}
\usepackage[all,2cell]{xy}
\xyoption{2cell}

\newtheorem{theorem}{Theorem}
\newtheorem{lemma}[theorem]{Lemma}
\newtheorem{corollary}[theorem]{Corollary}
\newtheorem{proposition}[theorem]{Proposition}

\newtheorem{example}[theorem]{Example}

\newtheorem{remark}[theorem]{Remark}
\usepackage[all]{xy}
\usepackage[active]{srcltx}
\usepackage[parfill]{parskip}

\newcommand{\tto}{\twoheadrightarrow}
\font\sc=rsfs10
\newcommand{\cC}{\sc\mbox{C}\hspace{1.0pt}}
\newcommand{\cQ}{\sc\mbox{Q}\hspace{1.0pt}}

\newcommand{\cS}{\sc\mbox{S}\hspace{1.0pt}}

\newcommand{\cA}{\sc\mbox{A}\hspace{1.0pt}}

\newcommand{\cX}{\sc\mbox{X}\hspace{1.0pt}}
\font\scc=rsfs7
\newcommand{\ccC}{\scc\mbox{C}\hspace{1.0pt}}

\newcommand{\ccA}{\scc\mbox{A}\hspace{1.0pt}}

\newcommand{\ccS}{\scc\mbox{S}\hspace{1.0pt}}

\newcommand{\ccQ}{\scc\mbox{Q}\hspace{1.0pt}}
\newcommand{\ccX}{\scc\mbox{X}\hspace{1.0pt}}

\begin{document}
\title[Analogues of centralizer subalgebras for fiat $2$-categories]
{Analogues of centralizer subalgebras for fiat $2$-categories and their $2$-representations}

\author{Marco Mackaay, Volodymyr Mazorchuk,\\ Vanessa Miemietz and Xiaoting Zhang}

\begin{abstract}
The main result of this paper establishes a bijection between the set of equivalence 
classes of simple transitive $2$-representations with
a fixed apex $\mathcal{J}$ of a fiat $2$-category $\cC$ and the set of equivalence 
classes of faithful simple transitive $2$-representations of the fiat $2$-subquotient of
$\cC$ associated with a diagonal $\mathcal{H}$-cell in $\mathcal{J}$.
As an application, we classify simple transitive $2$-representations of
various categories of Soergel bimodules, in particular, completing
the classification in types $B_3$ and $B_4$.
\end{abstract}

\maketitle

\section{Introduction and description of the results}\label{s1}

Let $\Bbbk$ be an algebraically closed field, $A$ a finite dimensional unital $\Bbbk$-algebra and 
$B$ a subalgebra of $A$, which is also assumed to be unital, however, whose unit is
not necessarily the same as in $A$. In this generality, there is no obvious relation 
between the categories $A$-mod and $B$-mod of finite dimensional $A$- and $B$-modules,
respectively. However, if we additionally assume that $B$ is the centralizer subalgebra 
of an idempotent $e$ in $A$, that is $B=eAe$, then the functor $Ae\otimes_{B}{}_-$
provides a full embedding of $B$-mod into $A$-mod.

Finitary $2$-categories, introduced in \cite{MM1}, are higher representation theoretic 
analogues of finite-dimensional algebras, where ``higher representation theory'' refers 
to the $2$-upgrade of the classical representation theory which originated in \cite{BFK,CR,KL,Ro}.
Foundations of $2$-representation theory of finitary $2$-categories were laid 
in the series \cite{MM1,MM2,MM3,MM4,MM5,MM6} and then various further developments appeared in 
\cite{CM,KMMZ,MMMT,MaMa,MT,MMZ1,MMZ2,MZ1,MZ2,Xa,Zh1,Zh2,Zi1,Zi2}, see \cite{Ma1,Ma2} for overviews.

The basic classification problem in higher representation theory is that of 
{\em simple transitive} $2$-representations of a given $2$-category $\cC$, as defined
in \cite{MM5}. These simple transitive $2$-representations are appropriate higher analogues of
simple modules over algebras. It turns out that in many cases 
(see \cite{MM5,MM6,GM1,GM2,Zh2,Zi1,Zi2,MZ1,MZ2,MMZ1,MMZ2} and \cite{Ma2} for an overview) 
simple transitive $2$-representations are exhausted by the class of {\em cell $2$-representations} 
which were defined already in \cite{MM1} using the intrinsic combinatorial structure of $\cC$. 
However, in many cases, the situation is more complicated or
the answer is not known.
Some elementary  examples of $2$-categories which have simple transitive 
$2$-representations that are not cell $2$-representations appear already in \cite{MM5}.
The first non-elementary example, a certain subquotient of the $2$-category of Soergel 
bimodules (over the coinvariant algebra) of type $B_2$, was considered 
and completed in \cite{MaMa}. There, both the construction 
and the classification rely on a delicate and technically very involved
analysis of the behavior of $2$-natural transformations and modifications related to a 
cell $2$-representation. The technique was later applied in \cite{KMMZ} to classify 
simple transitive $2$-representations for all {\em small quotients} of Soergel bimodules
associated to finite Coxeter systems. It turns out that simple transitive $2$-representations
which are not cell $2$-representations appear in all types $I_2(n)$, with $n>4$ even.

In the present paper we observe that the abstract results of \cite{MMMT} can, in some cases, be
applied to relate $2$-representations of a $2$-finitary category $\cC$ and a finitary $2$-subcategory
$\cA$ of $\cC$ of a certain kind. Very loosely, $\cA$ should be such that it
can be considered as a vague analogue of a ``centralizer subalgebra''
in $\cC$. Our analogy goes as follows: in the classical setup $eAe$-proj is equivalent to the 
idempotent split additive subcategory $\mathrm{add}(Ae)$ of $A$-proj. In our setup, 
$\cA$ is a $2$-subcategory of $\cC$ given as the ``additive span'' of some indecomposable
$1$-morphisms in $\cC$. One major complication in the $2$-setup compared to the classical case
is the absence of any kind of general induction procedure. In classical representation theory, 
one can always use induction given by the tensor product. So far, there is no reasonable general 
analogue of the tensor product in higher representation theory. We go around this difficulty and
use the approach of \cite{MMMT} to $2$-representations based on the study of (co)algebra 
$1$-morphisms
in the abelianization of $\cC$. This approach is based on ideas from \cite{EGNO,ENO,EO,Os}. 
Our main result is Theorem~\ref{thmmain} which can roughly be formulated as follows
(we refer the reader to the text for the precise formulation).

{\bf Main theorem.} For a fiat $2$-category $\cC$ and its $2$-subcategory $\cA$ as above,
there is a bijection between certain classes of simple transitive $2$-representations of
$\cC$ and $\cA$.

On the surface, this result carries a remarkable resemblance with the abstract 
classification of  simple modules over finite semigroups, see \cite{Cl,Mu,Po,GMS}.

Our main theorem has various immediate applications which work in both directions
(that is in the direction from $\cC$ to $\cA$ and 
in the direction from $\cA$ to $\cC$). In Section~\ref{s3} we use the
direction ``from $\cC$ to $\cA\,$'' to classify simple transitive $2$-representations of
certain subquotients of Soergel bimodules in finite Coxeter types, both simplifying and extending
the results of \cite{MaMa} and \cite{KMMZ}. In Section~\ref{s4} we use the
direction ``from $\cA$ to $\cC\,$'' to classify certain simple transitive $2$-representations 
of Soergel bimodules (over the coinvariant algebra) in finite Coxeter types, extending
the results of \cite{KMMZ} and, in particular, completing the classification in 
types $B_3$ and $B_4$. All preliminaries are collected in Section~\ref{s2}.
Section~\ref{snew1} contains some basic technical results on comodule categories over 
coalgebra  $1$-morphisms and functors between them.

\textbf{Acknowledgements:} 
The first author is partially supported by FCT/Portugal through project UID/MAT/04459/2013.
The second author is partially supported by
the Swedish Research Council, G{\"o}ran Gustafsson Stiftelse and Vergstiftelse. 
The fourth author is supported by G{\"o}ran Gustafsson Stiftelse. 

A part of this research was done during the visit, in October-November 2017, 
of the fourth author to the University of East Anglia, whose hospitality is gratefully acknowledged.
This visit was supported by Lisa och Carl-Gustav Esseens Stiftelse.
A part of this research was done during the visit, in November 2017, 
of all authors to the Hausdorff Institute of Mathematics in Bonn, 
whose hospitality and support are gratefully acknowledged.

We also thank Tobias Kildetoft for his help regarding the computation 
of cells in types $B_3$, $B_4$ and $B_5$.

\section{Preliminaries in abstract $2$-representation theory}\label{s2}

\subsection{$2$-categories and $2$-representations}\label{s2.1}

We refer to \cite{Le,McL} for generalities and terminology on $2$-categories.
We refer to \cite{MM1,MM2,MM3,MM4,MM5,MM6} for all terminology concerning 
$2$-representations of finitary $2$-categories. We recall some terminology below.

Following \cite{MM1}, a {\em finitary} $2$-category $\cC$ is a $2$-category such that
\begin{itemize}
\item it has finitely many objects;
\item each $\cC(\mathtt{i},\mathtt{j})$ is equivalent to $A_{\mathtt{i},\mathtt{j}}$-proj,
for some finite dimensional, unital, associative $\Bbbk$-algebra $A_{\mathtt{i},\mathtt{j}}$;
\item all compositions are (bi)additive and (bi)linear whenever applicable;
\item all identity $1$-morphisms are indecomposable.
\end{itemize}
Normally, we denote by $\circ_{\mathrm{h}}$ and $\circ_{\mathrm{v}}$ the horizontal and vertical compositions in $\cC$, respectively.
However, in Section~\ref{snew1} we will often suppress the notation $\circ_{\mathrm{h}}$ to save space in diagrams.
For the same reason, we will deviate from the notation of the \cite{MM1}-\cite{MM6} series
and denote by $1_{\mathrm{X}}$ the identity $2$-morphism of
a $1$-morphism $\mathrm{X}$.

A {\em finitary} $2$-representation of $\cC$ is a functorial action of $\cC$ on 
categories equivalent to $A$-proj, where $A$ is a finite dimensional, unital,
associative $\Bbbk$-algebra.
A typical example of a finitary $2$-rep\-re\-sen\-ta\-ti\-on is the {\em principal}
$2$-representation $\mathbf{P}_{\mathtt{i}}:=\cC(\mathtt{i},{}_-)$, where 
$\mathtt{i}\in\cC$. 
An {\em abelian} $2$-representation of $\cC$ is a functorial action of $\cC$ on 
categories  equivalent to $A$-mod, where $A$ is a finite dimensional, unital, associative $\Bbbk$-algebra.
For simplicity, given any $2$-representation $\mathbf{M}$, we will use the action notation $\mathrm{F}\,X$ instead of 
the representation notation $\mathbf{M}(\mathrm{F})(X)$.

Both finitary and abelian $2$-representations of $\cC$ form $2$-categories,
denoted $\cC$-afmod and $\cC$-mod, respectively, in which $1$-morphisms are strong 
$2$-natural transformations and $2$-morphisms are modifications, see \cite{MM3} for details. 
These two $2$-categories are connected by the diagrammatic {\em abelianization} $2$-functor
$\overline{\,\cdot\,}:\cC\text{-}\mathrm{afmod}\to \cC\text{-}\mathrm{mod}$,
see \cite[Subsection~4.2]{MM2}.

Let $\mathbf{M}\in\cC$-afmod. We say that $\mathbf{M}$ is {\em transitive}, see \cite[Subsection~3.1]{MM5},
provided that, for any $\mathtt{i},\mathtt{j}\in\cC$ and any indecomposable $X\in\mathbf{M}(\mathtt{i})$
and $Y\in\mathbf{M}(\mathtt{j})$, there exists $\mathrm{F}\in\cC(\mathtt{i},\mathtt{j})$ such that 
$Y$ is isomorphic to a summand of $\mathrm{F}\,X$. A transitive $2$-representation $\mathbf{M}$
is called {\em simple transitive}, see \cite[Subsection~3.5]{MM5}, if
$\mathbf{M}$ does not have any non-zero proper $\cC$-invariant ideals.
By \cite[Section~4]{MM5}, every $\mathbf{M}\in\cC$-afmod has a {\em weak Jordan-H{\"o}lder series}
with simple transitive subquotients which are defined uniquely up to permutation and equivalence.

If $\mathbf{M}$ is a $2$-representation of $\cC$, then the {\em rank} of $\mathbf{M}$ is the number of
isomorphism classes of indecomposable objects in $\displaystyle\coprod_{\mathtt{i}\in\ccC}\mathbf{M}(\mathtt{i})$.
Let $X_1,X_2,\dots,X_n$ be a fixed list of representatives of these isomorphism classes 
(in particular, $n$ is the rank of $\mathbf{M}$). To each $1$-morphism $\mathrm{F}$ one associates an 
$n\times n$-matrix $[\mathrm{F}]=(m_{i,j})$, where $m_{i,j}$ records the multiplicity of $X_i$ as a
summand of $\mathrm{F}\, X_j$.

\subsection{Cells and cell $2$-representations}\label{s2.2}

For indecomposable $1$-morphisms $\mathrm{F}$ and $\mathrm{G}$ in $\cC$, we write
$\mathrm{F}\geq_L \mathrm{G}$ provided that $\mathrm{F}$ is isomorphic to a direct summand of 
$\mathrm{H}\mathrm{G}$, for some $1$-morphism $\mathrm{H}$. The relation $\geq_L$
is called the {\em left} pre-order on the (finite) set of isomorphism classes of 
indecomposable $1$-morphisms in $\cC$. The {\em right} pre-order $\geq_R$ and the
{\em two-sided} pre-order $\geq_J$ are defined similarly using $\mathrm{G}\mathrm{H}$
and $\mathrm{H}_1\mathrm{G}\mathrm{H}_2$, respectively. The equivalence classes for 
these pre-orders are called the respective {\em cells}. We refer to \cite[Section~3]{MM2} 
for more details.

A two-sided cell is {\em regular} provided that all left (resp. right) cells
inside this two-sided cell are not comparable with respect to the left (resp. right) order.
A regular two-sided cell is called {\em strongly regular} if
the intersection of each left and each right cell inside this two-sided cell contains exactly
one isomorphism class of indecomposable $1$-morphisms.

Given a left cell $\mathcal{L}$ in $\cC$, there is $\mathtt{i}\in\cC$ such that all
$1$-morphisms in $\mathcal{L}$ start at $\mathtt{i}$. The additive closure, in $\coprod_{\mathtt{j}\in\ccC}\mathbf{P}_{\mathtt{i}}(\mathtt{j})$,
of all $1$-morphisms $\mathrm{F}$ satisfying $\mathrm{F}\geq_L\mathcal{L}$ is 
$\cC$-stable and hence gives a $2$-representation of $\cC$, denoted $\mathbf{N}_{\mathcal{L}}$.
The $2$-representation $\mathbf{N}_{\mathcal{L}}$ contains a unique ideal, denoted $\mathbf{I}_{\mathcal{L}}$, 
which is maximal in the set of all proper $\cC$-stable ideals of $\mathbf{N}_{\mathcal{L}}$. The
quotient $\mathbf{N}_{\mathcal{L}}/\mathbf{I}_{\mathcal{L}}$ is called the cell $2$-representation of 
$\cC$ associated to $\mathcal{L}$ and denoted $\mathbf{C}_{\mathcal{L}}$. Indecomposable 
objects of this $2$-representation are in natural bijection with indecomposable $1$-morphisms in 
$\mathcal{L}$. We refer to \cite[Subsection~6.5]{MM2}  for more details.

If $\mathbf{M}$ is any transitive $2$-representation of $\cC$, then there is a unique two-sided
cell, called the {\em apex} of $\mathbf{M}$, which is maximal, with respect to the two-sided order,
in the set of all two-sided cells whose elements are not annihilated by $\mathbf{M}$, 
see \cite[Subsection~3.2]{CM} for details.

\subsection{Fiat $2$-categories}\label{s2.3}

A finitary $2$-category $\cC$ is called {\em fiat} provided that it has a weak involutive 
anti-equivalence $\star$ (reversing the direction of both $1$- and $2$-morphisms) such that, for every
$1$-morphism $\mathrm{F}$, the pair $(\mathrm{F},\mathrm{F}^{\star})$ is an adjoint pair 
of $1$-morphisms. We refer to \cite[Subsection~2.4]{MM1}  for more details.

\subsection{Special modules for positively based algebras}\label{s2.51}

Here we assume that our field $\Bbbk$ is a subfield of $\mathbb{C}$. Let $A$ be a $\Bbbk$-algebra 
of dimension $n<\infty$ with a fixed basis $\mathbf{B}=\{1=a_1,\dots,a_n\}$. 
We say that $\mathbf{B}$ is {\em positive} if all structure constants with respect to this
basis are non-negative real numbers. For $i,j\in\{1,2,\dots,n\}$, set $a_i\leq_L a_j$ provided that
there is some $a_s$ such that $a_sa_i$ has a non-zero coefficient at $a_j$. This is the {\em left}
preorder on $\mathbf{B}$. Equivalence classes with respect to it are called {\em left cells}.
The right and two-sided preorders and cells are defined similarly using multiplication on the right
and on both sides, respectively.

An $A$-module $V$ with a fixed basis $\{v_1,\dots,v_k\}$ is said to be {\em positively based} provided
that the matrix of each $a_i$ in this basis has non-negative real coefficients and the
matrix of $a_1+a_2+\dots+ a_n$ has positive real coefficients. Each positively based
module has a unique {\em special} simple subquotient, namely the one which contain the 
Perron-Frobenius eigenvector for $a_1+a_2+\dots+ a_n$. Special simple modules for
positively based algebras correspond to non-nilpotent two-sided cells and 
are classified in \cite{KM}. We refer to \cite{KM} for more details.

\section{Coalgebra $1$-morphisms and simple transitivity of $2$-representations}\label{snew1}

\subsection{(Co)algebra $1$-morphisms and their (co)module categories}\label{s2.4}

Let $\cC$ be a finitary $2$-category. The paper \cite{MMMT} introduces certain, very technical,
injective and projective abelianizations  $\underline{\cC}$ and $\overline{\cC}$, respectively.
They have natural structures of $2$-categories extending the one on $\cC$. The main results of 
\cite{MMMT} require the assumption that $\cC$ is fiat, so we will also assume this from now on.
In case $\cC$ is fiat, its action on both $\underline{\cC}$ and $\overline{\cC}$ is given by
exact functors. As every object in $\underline{\cC}$ (resp. $\overline{\cC}$) is isomorphic to
a kernel (resp. cokernel) of a morphism between two objects in $\cC$, the snake lemma implies
that the action of $\underline{\cC}$ (resp. $\overline{\cC}$) on itself is left (resp. right) exact.

A {\em coalgebra $1$-morphism}  in $\underline{\cC}$  is a triple $(\mathrm{C},\Delta_{\mathrm{C}},\varepsilon)$, where 
$\mathrm{C}$ is a $1$-morphism and $\Delta_{\mathrm{C}}:\mathrm{C}\to\mathrm{C} \mathrm{C}$ 
and $\varepsilon: \mathrm{C}\to\mathbbm{1}$ are $2$-morphisms which satisfy the usual axioms for the comultiplication and counit of a coalgebra. Each coalgebra $1$-morphism $\mathrm{C}$ 
in $\underline{\cC}$ comes together with the category $\mathrm{comod}_{\underline{\ccC}}(\mathrm{C})$
of (right) $\mathrm{C}$-comodule $1$-morphisms in $\underline{\cC}$ and its full subcategory $\mathrm{inj}_{\underline{\ccC}}(\mathrm{C})$
of injective objects. The latter is a finitary (left) $2$-representation of $\cC$.
We also have the  category $(\mathrm{C})\mathrm{comod}_{\underline{\ccC}}$
of left $\mathrm{C}$-comodule $1$-morphisms in $\underline{\cC}$. 

Dually one defines {\em algebra $1$-morphisms} $(\mathrm{A},\mu,\eta)$ in $\overline{\cC}$
and the corresponding categories $\mathrm{mod}_{\overline{\ccC}}(\mathrm{A})$
and $(\mathrm{A})\mathrm{mod}_{\overline{\ccC}}$
of $\mathrm{A}$-module $1$-morphisms in $\overline{\cC}$ and its 
full subcategory $\mathrm{proj}_{\overline{\ccC}}(\mathrm{A})$
of projective objects. 
We refer to \cite{MMMT,EGNO} for further details.

\begin{lemma}\label{lemnnn1}
{\hspace{2mm}}

\begin{enumerate}[$($i$)$]
\item\label{lemnn1} The categories $\mathrm{comod}_{\underline{\ccC}}(\mathrm{C})$
and $(\mathrm{C})\mathrm{comod}_{\underline{\ccC}}$  are abelian subcategories of $\underline{\cC}$. 
\item\label{lemnn2} The categories $\mathrm{mod}_{\overline{\ccC}}(\mathrm{A})$
and $(\mathrm{A})\mathrm{mod}_{\overline{\ccC}}$  are abelian subcategories of $\overline{\cC}$. 
\end{enumerate}
\end{lemma}

\begin{proof}
The only non-obvious thing to check is the existence of kernels and cokernels. We do this for
claim~\eqref{lemnn1} and the dual construction works for claim~\eqref{lemnn2}.
Let $(\mathrm{M},\delta_{\mathrm{M}})$ and $(\mathrm{N},\delta_{\mathrm{N}})$ be two $\mathrm{C}$-comodule $1$-morphisms
and $\varphi:\mathrm{M}\to \mathrm{N}$ a homomorphism of comodule $1$-morphisms. Then we have the commutative
diagram
\begin{displaymath}
\xymatrix{
\mathrm{Ker}(\varphi)\ar@{^{(}->}[rr]^{\xi}\ar@{.>}[d]&&
\mathrm{M}\ar[rr]^{\varphi}\ar[d]_{\delta_{\mathrm{M}}}&&
\mathrm{N}\ar[d]_{\delta_{\mathrm{N}}}\ar@{->>}[rr]^{\zeta}&&\mathrm{Coker}(\varphi)\ar@{.>}[d]\\
\mathrm{Ker}(\varphi)\mathrm{C}\ar@{^{(}->}[rr]^{\xi1_{\mathrm{C}}}
&&\mathrm{M}\mathrm{C}\ar[rr]^{\varphi1_{\mathrm{C}}}&&
\mathrm{N}\mathrm{C}\ar[rr]^{\zeta1_{\mathrm{C}}}&&\mathrm{Coker}(\varphi)\mathrm{C}
}
\end{displaymath}
where the middle square commutes by definition, the upper row is exact and the 
lower row is left exact, since the action of $\underline{\cC}$ is left exact by construction. 
Using the universal property of kernels and cokernels,
we obtain the outer commuting squares that are induced by the middle one.
From this diagram it is easy to see that the kernel and the cokernel of $\varphi$ inherit the
structure of $\mathrm{C}$-comodule $1$-morphisms, as claimed.
\end{proof}

Theorem~9 in \cite{MMMT} asserts that, for every transitive $2$-representation $\mathbf{M}$ of a fiat
$2$-category $\cC$, there is a coalgebra $1$-morphism $\mathrm{C}$ in $\underline{\cC}$ such that 
$\mathbf{M}$ is equivalent to the $2$-representation $\mathrm{inj}_{\underline{\ccC}}(\mathrm{C})$ of $\cC$. Our main aim of this section is to prove that this correspondence matches
simple transitive $2$-representations and cosimple coalgebra $1$-morphisms, see Corollary~\ref{simple}.

\subsection{Simple transitive quotients and embeddings of coalgebra $1$-morphisms}\label{snew1.1}

Let $\cC$ be a fiat $2$-category.
Let $\mathbf{M}$ be a transitive $2$-representation of $\cC$
and $\mathbf{N}$ its simple transitive quotient.  Denote by 
$\pi$ the natural morphism $\mathbf{M} \twoheadrightarrow \mathbf{N}$.

Let $X$ be a non-zero object in $\mathbf{M}(\mathtt{i})$ (hence also, automatically,  
a non-zero object in $\mathbf{N}(\mathtt{i})$ as well). Denote by 
$\mathrm{C}=(\mathrm{C}_X,\Delta_{\mathrm{C}},\varepsilon_{\mathrm{C}})$ and 
$\mathrm{D}=(\mathrm{D}_X,\Delta_{\mathrm{D}},\varepsilon_{\mathrm{D}})$ 
the corresponding coalgebra $1$-morphisms in $\underline{\cC}$, given by
$$\mathrm{Hom}_{\underline{\mathbf{M}}(\mathtt{i})}(X, \mathrm{F}\, X) \cong  \mathrm{Hom}_{\underline{\ccC}}(\mathrm{C}, \mathrm{F}) \quad \hbox{
and } \quad\mathrm{Hom}_{\underline{\mathbf{N}}(\mathtt{i})}(X, \mathrm{F}\, X) \cong  \mathrm{Hom}_{\underline{\ccC}}(\mathrm{D}, \mathrm{F}),$$
respectively, for all $1$-morphisms $\mathrm{F} \in \underline{\cC}$.
We denote by $\mathrm{coev}_{X,X}^{\mathbf{M}}$ and $\mathrm{coev}_{X,X}^{\mathbf{N}}$
the image of $1_{\mathrm{C}}$ and $1_{\mathrm{D}}$ under the isomorphisms
$$\mathrm{Hom}_{\underline{\mathbf{M}}(\mathtt{i})}(X, \mathrm{C}\, X) \cong  \mathrm{Hom}_{\underline{\ccC}}(\mathrm{C}, \mathrm{C})\quad \hbox{
and } \quad\mathrm{Hom}_{\underline{\mathbf{N}}(\mathtt{i})}(X, \mathrm{D}\, X) \cong  \mathrm{Hom}_{\underline{\ccC}}(\mathrm{D}, \mathrm{D}),$$
respectively.

\begin{lemma}\label{alphabeta}
There exist $2$-morphisms $\alpha\colon \mathrm{D}\to \mathrm{C}$ and $\beta\colon \mathrm{C}\to \mathrm{D}$ in $\underline{\cC}$ such that the equality $\beta\alpha = 1_{\mathrm{D}}$ holds. 
In particular, $\alpha$ is a monomorphism.
\end{lemma}

\begin{proof}
The solid part of the diagram 
$$\xymatrix{
\mathrm{Hom}_{\underline{\mathbf{M}}(\mathtt{i})}(X, \mathrm{C}\, X) \ar^{\sim}[rr]\ar@{->>}[d]&& \mathrm{Hom}_{\underline{\ccC}}(\mathrm{C}, \mathrm{C})\ar@{-->}[d]\\
\mathrm{Hom}_{\underline{\mathbf{N}}(\mathtt{i})}(X, \mathrm{C}\, X) \ar^{\sim}[rr]&& \mathrm{Hom}_{\underline{\ccC}}(\mathrm{D}, \mathrm{C})\\
}$$
induces the dashed morphism on the right, and we define $\alpha$ to be the image of $1_{\mathrm{C}}$ under this morphism.

Similarly, the solid part of the diagram 
$$\xymatrix{
\mathrm{Hom}_{\underline{\mathbf{M}}(\mathtt{i})}(X, \mathrm{D}\, X) \ar^{\sim}[rr]\ar@{->>}[d]&& \mathrm{Hom}_{\underline{\ccC}}(\mathrm{C}, \mathrm{D})\ar@{-->}[d]\\
\mathrm{Hom}_{\underline{\mathbf{N}}(\mathtt{i})}(X, \mathrm{D}\, X) \ar^{\sim}[rr]&& \mathrm{Hom}_{\underline{\ccC}}(\mathrm{D}, \mathrm{D})\\
}$$
induces the dashed morphism on the right, and noting that this morphism is necessarily an epimorphism, we can choose a preimage $\beta \in \mathrm{Hom}_{\underline{\ccC}}(\mathrm{C}, \mathrm{D})$ of $1_{\mathrm{D}}\in \mathrm{Hom}_{\underline{\ccC}}(\mathrm{D}, \mathrm{D})$.

To prove $\beta\alpha = 1_{\mathrm{D}}$, consider the cube
$$\xymatrix{
\mathrm{Hom}_{\underline{\mathbf{M}}(\mathtt{i})}(X, \mathrm{C}\, X) \ar^{\sim}[rr]\ar@{->>}[dr]
\ar_{\beta_X\circ_{\mathrm{v}}{}_-}[dd]&& \mathrm{Hom}_{\underline{\ccC}}(\mathrm{C}, \mathrm{C})
\ar[dr]\ar^(.65){\beta\circ_{\mathrm{v}}{}_-}[dd]&\\
&\mathrm{Hom}_{\underline{\mathbf{N}}(\mathtt{i})}(X, \mathrm{C}\, X) \ar^{\sim}[rr]
\ar^(.65){\beta_X\circ_{\mathrm{v}}{}_-}[dd]&& \mathrm{Hom}_{\underline{\ccC}}(\mathrm{D}, \mathrm{C})
\ar^{\beta\circ_{\mathrm{v}}{}_-}[dd]\\
\mathrm{Hom}_{\underline{\mathbf{M}}(\mathtt{i})}(X, \mathrm{D}\, X) \ar^{\sim}[rr]\ar@{->>}[dr]&& \mathrm{Hom}_{\underline{\ccC}}(\mathrm{C}, \mathrm{D})\ar[dr]&\\
&\mathrm{Hom}_{\underline{\mathbf{N}}(\mathtt{i})}(X, \mathrm{D}\, X) \ar^{\sim}[rr]&& \mathrm{Hom}_{\underline{\ccC}}(\mathrm{D}, \mathrm{D})\\
}$$ 
which we claim commutes. Indeed, the top and bottom square commute by definition, the front and back arrows by functoriality of the isomorphisms defining the internal hom, and the left side square commutes since $\pi$ is a morphism of $2$-representations and to be precise , one would write $(\mathbf{M}(\beta))_X\circ_{\mathrm{v}}{}_-$ 
for the $\beta_X\circ_{\mathrm{v}}{}_-$ in the back and 
$(\mathbf{N}(\beta))_X\circ_{\mathrm{v}}{}_-$ for the one in the front, which equals 
$\pi(\mathbf{M}(\beta))_X\circ_{\mathrm{v}}{}_-$. Thus, the square on the right hand side commutes as well. Considering the image of $1_{\mathrm{C}}$ by going both ways around the square, we see that going down, this maps to $\beta$ and then to $1_{\mathrm{D}}$ going forwards, by definition of $\beta$. On the other hand, going forwards first, $1_{\mathrm{C}}$ maps to $\alpha $ by definition, and then going down to $\beta\alpha$. This proves $\beta\alpha=1_{\mathrm{D}}$.
\end{proof}

\begin{lemma}\label{commutes}
For any $1$-morphism $\mathrm{F}\in \underline{\cC}$, the diagram

$$\xymatrix{
\mathrm{Hom}_{\underline{\ccC}}(\mathrm{C}, \mathrm{F})\ar@{-}^{\ \sim}[d]\ar^{{}_-\circ_{\mathrm{v}}\alpha}[r]&\mathrm{Hom}_{\underline{\ccC}}(\mathrm{D}, \mathrm{F}) \ar@{-}^{\ \sim}[d]
\\
\mathrm{Hom}_{\underline{\mathbf{M}}(\mathtt{i})}(X, \mathrm{F}X)\ar@{->>}[r]^{\pi}&\mathrm{Hom}_{\underline{\mathbf{N}}(\mathtt{i})}(X, \mathrm{F}X)
}$$
commutes.
\end{lemma}

\begin{proof}
Taking any $\gamma\in \mathrm{Hom}_{\underline{\ccC}}(\mathrm{C}, \mathrm{F})$, the first vertical isomorphism sends $\gamma$ to the $2$-morphism $(\mathbf{M}(\gamma))_X\circ_{\mathrm{v}} \mathrm{coev}_{X,X}^{\mathbf{M}}$ and we have
\begin{displaymath}
\begin{aligned}
\pi((\mathbf{M}(\gamma))_X\circ_{\mathrm{v}}\mathrm{coev}_{X,X}^{\mathbf{M}})&=(\mathbf{N}(\gamma))_X\circ_{\mathrm{v}}\pi (\mathrm{coev}_{X,X}^{\mathbf{M}})\\
&=(\mathbf{N}(\gamma))_X\circ_{\mathrm{v}}(\mathbf{N}(\alpha))_X \circ_{\mathrm{v}}\mathrm{coev}_{X,X}^{\mathbf{N}}\\
&=(\mathbf{N}(\gamma\circ_{\mathrm{v}}\alpha))_X\circ_{\mathrm{v}}\mathrm{coev}_{X,X}^{\mathbf{N}}
\end{aligned}
\end{displaymath}
where the latter coincides with the image of $\gamma\circ_{\mathrm{v}}\alpha$ under the second vertical isomorphism.
\end{proof}

\begin{proposition}\label{coalgebramorphism}
The morphism $\alpha$ defined in Lemma \ref{alphabeta} is a morphism of coalgebra $1$-mor\-phisms.
\end{proposition}

\begin{proof}
Consider the diagram
$$\xymatrix{
\mathrm{Hom}_{\underline{\ccC}}(\mathrm{C}, \mathrm{C}) \ar^{{}_-\circ_{\mathrm{v}}\alpha}[r]\ar@{-}^{\sim}[d]
& \mathrm{Hom}_{\underline{\ccC}}(\mathrm{D}, \mathrm{C}) \ar@{-}^{\sim}[d]\ar^{\beta\circ_{\mathrm{v}}{}_-}[r]
& \mathrm{Hom}_{\underline{\ccC}}(\mathrm{D}, \mathrm{D})\ar@{-}^{\sim}[d]\ar^{\alpha\circ_{\mathrm{v}}{}_-}[l]<1ex>
\\ 
\mathrm{Hom}_{\underline{\mathbf{M}}(\mathtt{i})}(X, \mathrm{C}\, X) \ar@{->>}[r] 
\ar_{(1_{\mathrm{C}}\mathrm{coev}_{X,X}^{\mathbf{M}})\circ_{\mathrm{v}}{}_-}[d]
& \mathrm{Hom}_{\underline{\mathbf{N}}(\mathtt{i})}(X, \mathrm{C}\, X)\ar^{\beta_X\circ_{\mathrm{v}}{}_-}[r] \ar_{(1_{\mathrm{C}}(\alpha_X\circ_{\mathrm{v}}\mathrm{coev}_{X,X}^{\mathbf{N}}))\circ_{\mathrm{v}}{}_-}[d]& \mathrm{Hom}_{\underline{\mathbf{N}}(\mathtt{i})}(X, \mathrm{D}\, X)\ar^{\alpha_X\circ_{\mathrm{v}}{}_-}[l]<1ex>\ar^{(1_{\mathrm{D}}\mathrm{coev}_{X,X}^{\mathbf{N}})\circ_{\mathrm{v}}{}_-}[d]\\
\mathrm{Hom}_{\underline{\mathbf{M}}(\mathtt{i})}(X, \mathrm{C}\mathrm{C}\, X)  \ar@{->>}[r]\ar@{-}^{\sim}[d] 
& \mathrm{Hom}_{\underline{\mathbf{N}}(\mathtt{i})}(X, \mathrm{C}\mathrm{C}\, X) \ar@{-}^{\sim}[d] \ar^{((\beta\beta)_X)\circ_{\mathrm{v}}{}_-}[r]
& \mathrm{Hom}_{\underline{\mathbf{N}}(\mathtt{i})}(X, \mathrm{D}\mathrm{D}\, X)
\ar@{-}^{\sim}[d]\ar^{((\alpha\alpha)_X)\circ_{\mathrm{v}}{}_-}[l]<1ex>\\
\mathrm{Hom}_{\underline{\ccC}}(\mathrm{C}, \mathrm{C}\mathrm{C}) \ar^{{}_-\circ_{\mathrm{v}}\alpha}[r]
& \mathrm{Hom}_{\underline{\ccC}}(\mathrm{D}, \mathrm{C}\mathrm{C}) \ar^{(\beta\beta)\circ_{\mathrm{v}}{}_-}[r]
& \mathrm{Hom}_{\underline{\ccC}}(\mathrm{D}, \mathrm{D}\mathrm{D})\ar^{(\alpha\alpha)\circ_{\mathrm{v}}{}_-}[l]<1ex>\\ 
}$$ 
where the epimorphisms are induced by $\pi$, the vertical isomorphisms between the first and second, and the third and fourth row are the functorial isomorphisms from the definition of the internal hom.

The top two and the bottom two squares commute by construction (where for the squares on the right, commuting means that any two possible length two paths around the square are equal) and Lemma \ref{commutes}. To see that the left middle square commutes, note that under the epimorphism
$\mathrm{Hom}_{\underline{\mathbf{M}}(\mathtt{i})}(X, \mathrm{C}\, X)\twoheadrightarrow \mathrm{Hom}_{\underline{\mathbf{N}}(\mathtt{i})}(X, \mathrm{C}\, X)$ the image of $\mathrm{coev}_{X,X}^{\mathbf{M}}$ is $\alpha_X\circ_{\mathrm{v}} \mathrm{coev}_{X,X}^{\mathbf{N}}$ and that again $\pi$ is a morphism of $2$-representations. Commutativity of the right middle square (again meaning any paths of length two around the square are equal) is checked by direct computation.

To show $\alpha$ preserves comultiplication, we need to check that 
$(\alpha\circ_{\mathrm{h}}\alpha)\circ_{\mathrm{v}}\Delta_{\mathrm{D}}=\Delta_{\mathrm{C}}\circ_{\mathrm{v}}\alpha$. Equivalently, 
as $\beta$ is a left inverse to $\alpha$, we can show that 
\begin{equation}\label{comult}
\Delta_{\mathrm{D}}=(\beta\circ_{\mathrm{h}}\beta)\circ_{\mathrm{v}}\Delta_{\mathrm{C}}\circ_{\mathrm{v}}\alpha.                                                                                                                                                                                                                                                                                                                                                                                                                                                                                 \end{equation}

The right hand side of \eqref{comult} is precisely the image of the $2$-morphism
$1_{\mathrm{C}}\in\mathrm{Hom}_{\underline{\ccC}}(\mathrm{C}, \mathrm{C})$ in $\mathrm{Hom}_{\underline{\ccC}}(\mathrm{D}, \mathrm{D}\mathrm{D})$ by going down on the left and then across along the bottom line in the diagram (recall that the comultiplication $\Delta_{\mathrm{C}}$ is defined as the image, under the isomorphism with $\mathrm{Hom}_{\underline{\ccC}}(\mathrm{C}, \mathrm{C}\mathrm{C})$,  of $(1_{\mathrm{C}}\circ_{\mathrm{h}} \mathrm{coev}_{X,X}^{\mathbf{M}})\circ_{\mathrm{v}}\mathrm{coev}_{X,X}^{\mathbf{M}} \in \mathrm{Hom}_{\underline{\mathbf{M}}(\mathtt{i})}(X, \mathrm{C}\mathrm{C}\, X) $).

The left hand side of \eqref{comult} is the image of the $2$-morphism $1_{\mathrm{D}}\in\mathrm{Hom}_{\underline{\ccC}}(\mathrm{D}, \mathrm{D})$ when mapped to $\mathrm{Hom}_{\underline{\ccC}}(\mathrm{D}, \mathrm{D}\mathrm{D})$ along the right vertical of the diagram, so we are done if we can show that  $1_{\mathrm{C}}$ maps to $1_{\mathrm{D}}$ along the top of the diagram. However, the maps along the top take $1_{\mathrm{C}}$ to $\beta\alpha$, which equals $1_{\mathrm{D}}$ by Lemma \ref{alphabeta}, and thus $\alpha$ respects comultiplication.

Preservation of the counit simply follows from commutativity of the diagram
$$\xymatrix{
\mathrm{Hom}_{\underline{\mathbf{M}}(\mathtt{i})}(X, X) \ar@{-}^{\sim}[r]\ar@{->>}[d]
&\mathrm{Hom}_{\underline{\ccC}}(\mathrm{C}, \mathbbm{1}_{\mathtt{i}})\ar^{{}_-\circ_{\mathrm{v}}\alpha}[d]\\
\mathrm{Hom}_{\underline{\mathbf{N}}(\mathtt{i})}(X, X) \ar@{-}^{\sim}[r]
&\mathrm{Hom}_{\underline{\ccC}}(\mathrm{D}, \mathbbm{1}_{\mathtt{i}}),
}$$
which is obtained by applying Lemma \ref{commutes} to $\mathrm{F}=\mathbbm{1}_{\mathtt{i}}$.
This completes the proof.
\end{proof}

\subsection{Cotensor products}\label{snew1.43}
Let $\cC$ be a fiat $2$-category and $\mathrm{C}=(\mathrm{C},\tilde{\Delta},\tilde{\varepsilon})$ 
a coalgebra $1$-morphism in $\underline{\cC}$.
Let $(\mathrm{M},\delta_M)\in\mathrm{comod}_{\underline{\ccC}}(\mathrm{C})$ and
$(\mathrm{N},\delta_N)\in(\mathrm{C})\mathrm{comod}_{\underline{\ccC}}$.
In this subsection we adjust the classical construction of the
cotensor product $\mathrm{M}\square_{\mathrm{C}}\mathrm{N}$, see \cite{Ta}, to our setup.

We define $\mathrm{M}\square_{\mathrm{C}}\mathrm{N}$ as the equalizer of the
diagram
\begin{equation}\label{ndiag1}
\xymatrix{\mathrm{M}\square_{\mathrm{C}}\mathrm{N}\ar@{^{(}.>}[rr]^{\xi}&&
\mathrm{M}\mathrm{N}\ar@/^1pc/[rr]^{\delta_{\mathrm{M}}1_{\mathrm{N}}}
\ar@/_1pc/[rr]_{1_{\mathrm{M}}\delta_{\mathrm{N}}}&&
\mathrm{M}\mathrm{C}\mathrm{N},
} 
\end{equation}
that is, $\mathrm{M}\square_{\mathrm{C}}\mathrm{N}$ is defined as the kernel of the $2$-morphism
$\delta_{\mathrm{M}}1_{\mathrm{N}}-1_{\mathrm{M}}\delta_{\mathrm{N}}$. 
For a fixed $\mathrm{M}$, this is, obviously,
covariantly functorial in  $\mathrm{N}$, and vice versa. It follows that
${}_-\square_{\mathrm{C}}{}_-$ is a bifunctor.

Assume that $\mathrm{D}=(\mathrm{D},\Delta,\varepsilon)$ is another 
coalgebra $1$-morphism in $\underline{\cC}$ and that $\mathrm{N}$ is a 
$\mathrm{C}$-$\mathrm{D}$-bicomodule $1$-morphism with the right coaction of $\mathrm{D}$ given
by $\gamma:\mathrm{N}\to\mathrm{N}\mathrm{D}$. Then $\mathrm{M}\square_{\mathrm{C}}\mathrm{N}$
has the natural structure of a right $\mathrm{D}$-comodule $1$-morphism, as seen from the 
diagram
\begin{equation}\label{ndiag2}
\xymatrix{\mathrm{M}\square_{\mathrm{C}}\mathrm{N}\ar@{^{(}.>}[rr]^{\xi}
\ar@{.>}[dd]_{\delta_{\mathrm{M}\square_{\mathrm{C}}\mathrm{N}}}&&
\mathrm{M}\mathrm{N}\ar@/^1pc/[rr]^{\delta_{\mathrm{M}}1_{\mathrm{N}}}
\ar@/_1pc/[rr]_{1_{\mathrm{M}}\delta_{\mathrm{N}}}\ar[dd]_{1_{\mathrm{M}}\gamma}&&
\mathrm{M}\mathrm{C}\mathrm{N}\ar[dd]^{1_{\mathrm{M}}1_{\mathrm{C}}\gamma}\\\\
\mathrm{M}\square_{\mathrm{C}}\mathrm{N}\mathrm{D}\ar@{^{(}.>}[rr]^{\xi1_{\mathrm{D}}}&&
\mathrm{M}\mathrm{N}\mathrm{D}\ar@/^1pc/[rr]^{\delta_{\mathrm{M}}1_{\mathrm{N}}1_{\mathrm{D}}}
\ar@/_1pc/[rr]_{1_{\mathrm{M}}\delta_{\mathrm{N}}1_{\mathrm{D}}}&&
\mathrm{M}\mathrm{C}\mathrm{N}\mathrm{D}.
} 
\end{equation}
Here we have two commutative solid squares: the vertical solid maps form a commutative square
with the upper horizontal solid maps by interchange law, and with the lower  
horizontal solid maps due to the fact that $\mathrm{N}$ is a 
$\mathrm{C}$-$\mathrm{D}$-bicomodule $1$-morphism. Therefore we obtain the induced dotted map on the right
which defines the necessary structure. Due to functoriality of the construction mentioned above,
we thus have a functor
\begin{displaymath}
{}_-\square_{\mathrm{C}}\mathrm{N}: \mathrm{comod}_{\underline{\ccC}}(\mathrm{C})\to 
\mathrm{comod}_{\underline{\ccC}}(\mathrm{D}). 
\end{displaymath}
The following lemma establishes one of the basic familiar properties of this functor.

\begin{lemma}\label{lemnnn2}
The functor  
\begin{displaymath}
{}_-\square_{\mathrm{C}}\mathrm{C}: \mathrm{comod}_{\underline{\ccC}}(\mathrm{C})\to 
\mathrm{comod}_{\underline{\ccC}}(\mathrm{C})
\end{displaymath}
is isomorphic to the identity functor on $\mathrm{comod}_{\underline{\ccC}}(\mathrm{C})$.
\end{lemma}

\begin{proof}
Let $(\mathrm{M},\delta)$ be a $\mathrm{C}$-comodule $1$-morphism. Consider the diagram 
\begin{displaymath}
\xymatrix{\mathrm{M}\ar[rr]^{\delta}&&
\mathrm{M}\mathrm{C}\ar@/^1pc/[rr]^{\delta1_{\mathrm{C}}}
\ar@/_1pc/[rr]_{1_{\mathrm{M}}\tilde{\Delta}}&&
\mathrm{M}\mathrm{C}\mathrm{C}.
} 
\end{displaymath}
The left morphism on this diagram equalizes the right two morphisms due to the coassociativity 
part of the coaction axiom. By the counitality part of the coaction axiom, the composite
\begin{displaymath}
\xymatrix{\mathrm{M}\ar[rr]^{\delta}&&
\mathrm{M}\mathrm{C}\ar[rr]^{1_{\mathrm{M}}\tilde{\varepsilon}}&&\mathrm{M}.
} 
\end{displaymath}
is the identity. Therefore $\delta$ is monic. By the universal property of kernels, it follows
that $\delta$ factors through $\mathrm{M}\square_{\mathrm{C}}\mathrm{C}$ and the factorization map
is a monomorphism. This thus defines a monic natural transformation from the identity functor 
to ${}_-\square_{\mathrm{C}}\mathrm{C}$ and it remains to verify that it is an isomorphism.
For this it is enough to construct a monomorphism from $\mathrm{M}\square_{\mathrm{C}}\mathrm{C}$
to $\mathrm{M}$.

Consider the diagram
\begin{displaymath}
\xymatrix{
\mathrm{M}\square_{\mathrm{C}}\mathrm{C}\ar[rr]^{\xi}&&
\mathrm{M}\mathrm{C}\ar@/^1pc/[rr]^{\delta1_{\mathrm{C}}}
\ar@/_1pc/[rr]_{1_{\mathrm{M}}\tilde{\Delta}}\ar[d]_{1_{\mathrm{M}}\tilde{\varepsilon}}&&
\mathrm{M}\mathrm{C}\mathrm{C}\ar[d]^{1_{\mathrm{M}}1_{\mathrm{C}}\tilde{\varepsilon}}\\
&&\mathrm{M}\ar[rr]_{\delta}&&\mathrm{M}\mathrm{C}
} 
\end{displaymath}
With respect to the upper of the two arrows from $\mathrm{M}\mathrm{C}$ to $\mathrm{M}\mathrm{C}\mathrm{C}$,
we have a commutative square by the interchange law. The composite of the lower of the two arrows
with the rightmost vertical arrow is the identity by the counitality axiom. As the arrow starting at
$\mathrm{M}\square_{\mathrm{C}}\mathrm{C}$ equalizes these two arrows, all paths from the north-west
$\mathrm{M}\square_{\mathrm{C}}\mathrm{C}$ to the south-east $\mathrm{M}\mathrm{C}$ on the diagram are equal to $\xi$. 
In particular, 
$(1_{\mathrm{M}}\tilde{\varepsilon})\circ_{\mathrm{v}} \xi:\mathrm{M}\square_{\mathrm{C}}\mathrm{C}\to \mathrm{M}$ 
is a right factor of
a monomorphism and hence is monic. The claim follows.
\end{proof}

\subsection{An adjunction using cotensor products}\label{snew1.45}

Let $\cC$ be a fiat $2$-category. Consider two coalgebra $1$-morphisms  
$\mathrm{C}=(\mathrm{C},\tilde{\Delta},\tilde{\varepsilon})$ and 
$\mathrm{D}=(\mathrm{D},\Delta,\varepsilon)$ 
in $\underline{\cC}$. Consider also a monomorphism $\iota:\mathrm{D}\to\mathrm{C}$ 
of coalgebra $1$-morphisms. In this situation we have an exact embedding 
\begin{displaymath}
\Psi:\mathrm{comod}_{\underline{\ccC}}(\mathrm{D})\longrightarrow\mathrm{comod}_{\underline{\ccC}}(\mathrm{C})
\end{displaymath}
where, given a $\mathrm{D}$-comodule $1$-morphism $(\mathrm{M},\delta_{\mathrm{M}})$, the 
$\mathrm{C}$-comodule $1$-morphism $\Psi(\mathrm{M})$ is defined to be $\mathrm{M}$ with the coaction
map given by the composite
\begin{displaymath}
\xymatrix{
\mathrm{M}\ar[rr]^{\delta_{\mathrm{M}}}&&\mathrm{M}\mathrm{D}
\ar[rr]^{1_{\mathrm{M}}\iota}&&\mathrm{M} \mathrm{C}.
}
\end{displaymath}

The monomorphism $\iota$ defines on $\mathrm{D}$ the structure of a left $\mathrm{C}$-comodule $1$-morphism via
\begin{displaymath}
\xymatrix{
\mathrm{D}\ar[rr]^{\Delta_{\mathrm{D}}}&&\mathrm{D}\mathrm{D}
\ar[rr]^{\iota1_{\mathrm{D}}}&&\mathrm{C}\mathrm{D} ,
}
\end{displaymath}
which, in fact, turns $\mathrm{D}$ into a $\mathrm{C}$-$\mathrm{D}$-bicomodule $1$-morphism
with the regular right coaction of $\mathrm{D}$. The aim of this subsection is to show 
that the functor $\Phi:={}_-\square_{\mathrm{C}}\mathrm{D}$ is right adjoint to $\Psi$
and to describe some basic properties of this adjunction. All diagrams in this section will be diagrams in $\underline{\cC}$.

For a $\mathrm{C}$-comodule $1$-morphism  $\mathrm{M}$, we denote by $\xi_{\mathrm{M}}$ the embedding of
$\mathrm{M}\square_{\mathrm{C}}\mathrm{D}$ into $\mathrm{M}\mathrm{D}$ given by the definition
of ${}_-\square_{\mathrm{C}}\mathrm{D}$.

\begin{lemma}\label{lemmarco1}
Let $(\mathrm{M},\delta)$ be a $\mathrm{C}$-comodule $1$-morphism. Then the composite
\begin{displaymath}
\xymatrix{
\mathrm{M}\square_{\mathrm{C}}\mathrm{D}\ar@{^{(}->}[rr]^{\xi_{\mathrm{M}}}&&
{\mathrm{M}}\mathrm{D}\ar[rr]^{1_{\mathrm{M}}\varepsilon}&&{\mathrm{M}}
}
\end{displaymath}
is a monomorphism.
\end{lemma}

We denote the composite monomorphism given by Lemma~\ref{lemmarco1} by $\zeta_{{\mathrm{M}}}$.

\begin{proof}
Consider the diagram
\begin{equation}\label{diag38}
\xymatrix{
\mathrm{M}\mathrm{C}&&\mathrm{M}\mathrm{C}\mathrm{D}\ar[ll]_{1_\mathrm{M}1_{\mathrm{C}}\varepsilon}&&
\mathrm{M}\mathrm{D}\mathrm{D}\ar[rr]^{1_\mathrm{M}1_{\mathrm{D}}\varepsilon}\ar[ll]_{1_\mathrm{M}\iota1_{\mathrm{D}}}
&&\mathrm{M}\mathrm{D}\ar@{=}[d]\ar@/_3pc/[llllll]_{1_\mathrm{M}\iota}
\\
&&&&&&\mathrm{M}\mathrm{D}\ar[llu]_{1_\mathrm{M}{\Delta}}
}
\end{equation}
Here the upper part commutes by the interchange law and the lower part commutes
by the axioms of a coalgebra. Hence this diagram commutes.

Now consider the diagram
\begin{equation}\label{diag39}
\xymatrix{
\mathrm{M}\square_{\mathrm{C}}\mathrm{D}\ar@{^{(}->}[rrrr]^{\xi_\mathrm{M}}&&&&
\mathrm{M}\mathrm{D}\ar[rr]^{1_\mathrm{M}\varepsilon}\ar[d]^{\delta1_{\mathrm{D}}}\ar@{=}[lllld]\ar[lld]^{1_m\Delta}&&
\mathrm{M}\ar[d]^{\delta}\\
\mathrm{M}\mathrm{D}\ar@/_3pc/@{_{(}->}[rrrrrr]^{1_\mathrm{M}\iota}&&
\mathrm{M}\mathrm{D}\mathrm{D}\ar[rr]_{1_\mathrm{M}\iota1_{\mathrm{D}}}\ar[ll]^{1_\mathrm{M}1_{\mathrm{D}}\varepsilon}&&
\mathrm{M}\mathrm{C}\mathrm{D}\ar[rr]_{1_\mathrm{M}1_{\mathrm{C}}\varepsilon}&&\mathrm{M}\mathrm{C}\\
}
\end{equation}
The part of this diagram which can be recognized as \eqref{diag38} commutes by the previous paragraph.
The right square commutes by the interchange law. And the map $\xi_\mathrm{M}$ is, by definition, the equalizer of
the middle triangle. Therefore, if we start at $\mathrm{M}\square_{\mathrm{C}}\mathrm{D}$, 
then the whole diagram commutes.

In particular, the composite $\delta\circ_{\mathrm{v}} (1_\mathrm{M}\varepsilon)\circ_{\mathrm{v}} \xi_\mathrm{M}$ equals the
composite of the monomorphism $\xi_\mathrm{M}$ and the monomorphism $1_\mathrm{M}\iota$. Hence 
$\delta\circ_{\mathrm{v}} (1_\mathrm{M}\varepsilon)\circ_{\mathrm{v}} \xi_\mathrm{M}$ is a monomorphism. Consequently,
$(1_\mathrm{M}\varepsilon)\circ_{\mathrm{v}} \xi_\mathrm{M}$ is a monomorphism, as claimed.
\end{proof}

As an immediate corollary from Lemma~\ref{lemmarco1}, we have:

\begin{corollary}\label{cormarco2}
The functor $\Phi\Psi$ is isomorphic to the identity functor on 
$\mathrm{comod}_{\underline{\ccC}}(\mathrm{D})$. 
\end{corollary}

\begin{proof}
Let $(\mathrm{M},\delta)$ be a $\mathrm{D}$-comodule $1$-morphism.
Consider the diagram 
\begin{displaymath}
\xymatrix{
\mathrm{M}\mathrm{D}\ar@/^/[rr]^{\delta1_{\mathrm{D}}}\ar@/_/[rr]_{1_\mathrm{M}\Delta}
&&\mathrm{M}\mathrm{D}\mathrm{D}\ar[rr]^{1_\mathrm{M}\iota1_{\mathrm{D}}}&&\mathrm{M}\mathrm{C}\mathrm{D}
}
\end{displaymath}
By definition, $\mathrm{M}\square_{\mathrm{D}}\mathrm{D}$ is the equalizer of the two
morphisms on the left and $\mathrm{M}\square_{\mathrm{C}}\mathrm{D}$ is the equalizer
of the two morphisms from the left to the right. As 
$\iota$ is monic and composition in $\underline{\cC}$ is left exact by construction, 
we get that  $1_\mathrm{M}\iota1_{\mathrm{D}}$ is monic.
Hence $\mathrm{M}\square_{\mathrm{C}}\mathrm{D}\cong \mathrm{M}\square_{\mathrm{D}}\mathrm{D}$
and the latter is isomorphic to $\mathrm{M}$ by Lemma~\ref{lemnnn2}. The claim follows.
\end{proof}

\begin{lemma}\label{lemmarco3}
There is a monic natural transformation from $\Psi\Phi$ to the identity functor on 
$\mathrm{comod}_{\underline{\ccC}}(\mathrm{C})$. 
\end{lemma}

\begin{proof}
Let $(\mathrm{M},\delta)$ be a $\mathrm{C}$-comodule $1$-morphism.
Consider the diagram
\begin{equation}\label{diag51}
\xymatrix{
&&\mathrm{M}\ar[rr]^{\delta}&&\mathrm{M}\mathrm{C}\\
\Phi(\mathrm{M})\ar[rr]^{\xi_\mathrm{M}}\ar[d]&&
\mathrm{M}\mathrm{D}\ar[u]^{1_\mathrm{M}\varepsilon}\ar[rr]^{\delta1_{\mathrm{D}}}\ar[d]_{1_\mathrm{M}\Delta}
&&\mathrm{M}\mathrm{C}\mathrm{D}\ar[u]^{1_\mathrm{M}1_{\mathrm{C}}\varepsilon}
\ar[dd]_{1_{\mathrm{M}}1_{\mathrm{C}}\iota}\\
\Phi(\mathrm{M})\ar[rr]^{\xi_\mathrm{M}1_{\mathrm{D}}}\mathrm{D}\ar[d]_{1_{\Phi(\mathrm{M})}\iota}&&
\mathrm{M}\mathrm{D}\mathrm{D}\ar[d]_{1_\mathrm{M}1_{\mathrm{D}}\iota}\ar[rru]^{1_\mathrm{M}\iota1_{\mathrm{D}}}&&\\
\Phi(\mathrm{M})\ar[rr]^{\xi_\mathrm{M}1_{\mathrm{C}}}\mathrm{C}&&
\mathrm{M}\mathrm{D}\mathrm{C}\ar[rr]^{1_\mathrm{M}\iota1_{\mathrm{C}}}&&
\mathrm{M}\mathrm{C}\mathrm{C}.\ar@/_2pc/[uuu]_{1_\mathrm{M}1_{\mathrm{C}}\tilde{\varepsilon}}
}
\end{equation}
We claim that everything, except for the triangle in the middle of this diagram, commutes.
The top square in the first column commutes by the definition of $\Phi$. The
bottom square in the first column commutes by the interchange law. The rightmost triangle
commutes as $\iota$ is a homomorphism of coalgebra $1$-morphisms. The top square
in the middle column commutes by the interchange law. The lower quadrilateral in the middle
column also commutes by the interchange law. 

Starting with $\xi_\mathrm{M}$, the triangle in the middle commutes by the definition of $\Phi$. This means
that all paths in \eqref{diag51} starting at $\Phi(\mathrm{M})$ and ending in $\mathrm{M}\mathrm{C}$ commute.

Next we consider the diagram
\begin{equation}\label{diag52}
\xymatrix{
\mathrm{D}\ar[d]_{\iota}\ar[rr]^{\Delta}&&
\mathrm{D}\mathrm{D}\ar[d]_{\iota\iota}\ar@/^/[rr]^{\varepsilon1_{\mathrm{D}}}
\ar@/_/[rr]_{1_{\mathrm{D}}\varepsilon}&&
\mathrm{D}\ar[d]_{\iota}\\
\mathrm{C}\ar[rr]^{\tilde{\Delta}}&&
\mathrm{C}\mathrm{C}\ar@/^/[rr]^{\tilde{\varepsilon}1_{\mathrm{C}}}
\ar@/_/[rr]_{1_{\mathrm{C}}\tilde{\varepsilon}}&&
\mathrm{C}.
}
\end{equation}
This diagram commutes since $\iota$ is a homomorphism of coalgebra $1$-morphisms (in fact, all
horizontal paths from the extreme left all the way to the extreme right are the identities).
From \eqref{diag52} it follows that, in the path
\begin{displaymath}
\mathrm{M}\mathrm{D}\to \mathrm{M}\mathrm{D}\mathrm{D}\to  
\mathrm{M}\mathrm{D}\mathrm{C}\to \mathrm{M}\mathrm{C}\mathrm{C}\to \mathrm{M}\mathrm{C}
\end{displaymath}
in \eqref{diag51},
the very last step $1_\mathrm{M}1_{\mathrm{C}}\tilde{\varepsilon}$ 
can be exchanged with $1_\mathrm{M}\tilde{\varepsilon}1_{\mathrm{C}}$ without
changing the whole composite.

The latter, in turn, implies that $(1_\mathrm{M}\varepsilon)\circ_{\mathrm{v}} \xi_\mathrm{M}$ 
is a natural transformation
from $\Psi\Phi$ to the identity functor on 
$\mathrm{comod}_{\underline{\ccC}}(\mathrm{C})$.
It is monic by Lemma~\ref{lemmarco1}, as claimed.
\end{proof}

\begin{corollary}\label{cormarco4}
The functor $\Phi$ is right adjoint to $\Psi$. 
\end{corollary}

\begin{proof}
Let $\alpha$ denote the natural transformation from the identity 
functor on $\mathrm{comod}_{\underline{\ccC}}(\mathrm{D})$
to $\Phi\Psi$ constructed in Corollary~\ref{cormarco2}.
Let $\beta$ denote the natural transformation from 
$\Psi\Phi$ to the identity functor on $\mathrm{comod}_{\underline{\ccC}}(\mathrm{C})$
constructed in Lemma~\ref{lemmarco3}.

Let $(\mathrm{M},\delta_\mathrm{M})$ be a  $\mathrm{D}$-comodule $1$-morphism and $(\mathrm{N},\delta_\mathrm{N})$ 
a  $\mathrm{C}$-comodule $1$-morphism. 
We have to check that both composites
\begin{equation}\label{eq100}
\Phi(\mathrm{N})\overset{\alpha_{\Phi(\mathrm{N})}}{\longrightarrow}\Phi\Psi\Phi(\mathrm{N})
\overset{\Phi(\beta_\mathrm{N})}{\longrightarrow}\Phi(\mathrm{N})\quad\text{ and }\quad
\Psi(\mathrm{M})\overset{\Psi(\alpha_\mathrm{M})}{\longrightarrow}\Psi\Phi\Psi(\mathrm{M})
\overset{\beta_{\Psi(\mathrm{M})}}{\longrightarrow}\Psi(\mathrm{M})
\end{equation}
are the identities. 

By definition of $\alpha$, we have the commutative diagram
\begin{equation}\label{diag101}
\xymatrix{
\mathrm{M}\ar[rr]^{\delta_{\mathrm{M}}}\ar[d]_{\alpha_\mathrm{M}}&&\mathrm{M}\mathrm{D}\ar@{=}[d]\\
\Phi\Psi(\mathrm{M})\ar[rr]^{\xi_{\Psi(\mathrm{M})}}&&\Psi(\mathrm{M})\mathrm{D}.
}
\end{equation}
Applying $\Psi$ to the left vertical arrow gives the commutative diagram
\begin{equation}\label{diag102}
\xymatrix{
\Psi(\mathrm{M})\ar[rr]^{\delta_{\mathrm{M}}}\ar[d]_{\alpha_\mathrm{M}=
\Psi(\alpha_\mathrm{M})}&&\mathrm{M}\mathrm{D}\ar@{=}[d]\\
\Psi\Phi\Psi(\mathrm{M})\ar[rr]^{\xi_{\Psi(\mathrm{M})}}&&\Psi(\mathrm{M})\mathrm{D}.
}
\end{equation}
Similarly, by the definition of $\beta$, we have the commutative diagram
\begin{equation}\label{diag103}
\xymatrix{
\Psi\Phi(\mathrm{N})\ar[rr]^{\xi_\mathrm{N}}\ar[drr]_{\beta_\mathrm{N}}&&\mathrm{N}\mathrm{D}
\ar[d]^{1_\mathrm{N}\varepsilon}\\
&&\mathrm{N}
}
\end{equation}
For $\mathrm{N}=\Psi(\mathrm{M})$, we get the commutative diagram
\begin{equation}\label{diag104}
\xymatrix{
\Psi\Phi\Psi(\mathrm{M})\ar[rr]^{\xi_{\Psi(\mathrm{M})}}\ar[drr]_{\beta_{\Psi(\mathrm{M})}}&&
\Psi(\mathrm{M})\mathrm{D}\ar[d]^{1_{\Psi(\mathrm{M})}\varepsilon}\\
&&\Psi(\mathrm{M})
}
\end{equation}
Putting \eqref{diag102} on top of  \eqref{diag104} and using the counitality axiom for $\mathrm{M}$
verifies  that the second composite in \eqref{eq100} is the identity.

To verify that the first composite in \eqref{eq100} is the identity, consider the diagram
\begin{equation}\label{diag105}
\xymatrix{
&&\mathrm{N}\mathrm{D}\ar[drr]^{1_\mathrm{N}\Delta}&&\\
\Phi(\mathrm{N})\ar^{\Phi(\delta_{\mathrm{N}})}[rr]\ar[d]_{\alpha_{\Phi(\mathrm{N})}}\ar^{\xi_\mathrm{N}}[urr]&&
\Phi(\mathrm{N})\mathrm{D}\ar^{\xi_\mathrm{N} 1_{\mathrm{D}}}[rr] &&
\mathrm{N}\mathrm{DD}\ar^{1_\mathrm{N}\varepsilon 1_D}[d]\\
\Phi\Psi\Phi(\mathrm{N})\ar[urr]^{\xi_{\Psi\Phi(\mathrm{N})}}\ar[rr]^{\Phi(\beta_\mathrm{N})}&&
\Phi(\mathrm{N})\ar^{\xi_\mathrm{N}}[rr]&&\mathrm{N}\mathrm{D}.
}
\end{equation}
Here the left triangle is just \eqref{diag101} evaluated at $\Phi(\mathrm{N})$ and hence  commutes.
Note that $\beta_\mathrm{N}:=(1_\mathrm{N}\varepsilon)\circ_{\mathrm{v}} \xi_\mathrm{N} $ by \eqref{diag103}. 
The pentagon commutes due to the definition of how $\Phi$ acts on 
morphisms, see the paragraph just before Lemma~\ref{lemmarco1}. 
Indeed, from the definition of  $\beta_\mathrm{N}$ it follows that
the composite $\Phi(\mathrm{N})\mathrm{D}\to \mathrm{N}\mathrm{DD}\to \mathrm{N}\mathrm{D}$ 
is $\beta_{\mathrm{N}}1_{\mathrm{D}}$. 
Applying the definition of how $\Phi$ acts on $\beta_\mathrm{N}$, we  get the
commutativity of the pentagon. The top part of the diagram commutes by \eqref{ndiag2}.
Therefore the whole diagram \eqref{diag105} commutes.

Consequently,  we have
\begin{displaymath}
\xi_\mathrm{N}\circ_{\mathrm{v}}\Phi(\beta_\mathrm{N})\circ_{\mathrm{v}}\alpha_{\Phi(\mathrm{N})}=
(1_\mathrm{N}\varepsilon1_\mathrm{D})\circ_{\mathrm{v}}(1_\mathrm{N}\Delta)\circ_{\mathrm{v}}\xi_\mathrm{N}=\xi_\mathrm{N},
\end{displaymath}
where the last equality follows from the counitality axiom for $\mathrm{D}$. 
Using monicity of $\xi_\mathrm{N}$, this verifies that the  first composite in \eqref{eq100}
is the identity morphism as well. This completes the proof.
\end{proof}

As an immediate consequence, we may now record:

\begin{corollary}\label{cormarco5}
The functor $\Phi$ maps injective comodule $1$-morphisms to injective comodule $1$-morphisms and is full on injective comodule $1$-morphisms.
\end{corollary}

\begin{proof}
By Corollary~\ref{cormarco4}, $\Phi$ is right adjoint to an exact functor. 
Therefore $\Phi$ maps injective comodule $1$-morphisms to injective comodule $1$-morphisms. Let $\mathrm{X}, \mathrm{Y}\in \mathrm{inj}_{\underline{\ccC}}(\mathrm{C})$.
Let $\varphi:\Phi(\mathrm{X})\to \Phi(\mathrm{Y})$ be a homomorphism. Consider
$\Psi(\varphi):\Psi\Phi(\mathrm{X})\to \Psi\Phi(\mathrm{Y})$. We can now use Lemma~\ref{lemmarco3}
to embed $\Psi\Phi(\mathrm{Y})$ into $\mathrm{Y}$ and then use the injectivity of $\mathrm{Y}$ to
lift $\Psi(\varphi)$ to a morphism $\psi:\mathrm{X}\to \mathrm{Y}$. Then $\Phi(\psi)=\varphi$ by 
Corollary~\ref{cormarco2}. This completes the proof.
\end{proof}

\begin{proposition}\label{propmarco6}
The functor $\Phi$ is a homomorphism of $2$-representations of $\cC$.
\end{proposition}

\begin{proof}
Taking into account the exactness of the action of $\cC$,
the assertion follows from the construction and 
the fact that $\Psi$ is a strict  homomorphism of $\cC$-modules.
Alternatively, one can argue as \cite{Ke} that $\Phi$ has a canonical structure of a
homomorphism of $\cC$-modules
\begin{equation}\label{diag59}
\mathrm{F}\Phi(\mathrm{Y})\cong \Phi\Psi(\mathrm{F}\Phi(\mathrm{Y}))=
\Phi(\mathrm{F}\Psi\Phi(\mathrm{Y}))\to\Phi(\mathrm{F}\,\mathrm{Y}) 
\end{equation}
given by adjunction and the fact that $\Psi$ is a strict homomorphism of 
$2$-representations of $\cC$. Lemma~\eqref{lemmarco3}
implies that the right map in \eqref{diag59}
is monic. For $\mathrm{X}\in \mathrm{inj}_{\underline{\ccC}}(\mathrm{D})$, we have
\begin{displaymath}
\begin{array}{ccc}
\mathrm{Hom}_{\mathrm{inj}_{\underline{\ccC}}(\mathrm{D})}(\mathrm{X},\Phi(\mathrm{F}\,\mathrm{Y}))
&\cong&\mathrm{Hom}_{\mathrm{inj}_{\underline{\ccC}}(\mathrm{D})}(\Psi(\mathrm{X}),\mathrm{F}\,\mathrm{Y})\\
&\cong&\mathrm{Hom}_{\mathrm{inj}_{\underline{\ccC}}(\mathrm{D})}(\mathrm{F}^{\star}\,\Psi(\mathrm{X}),\mathrm{Y})\\
&\cong&\mathrm{Hom}_{\mathrm{inj}_{\underline{\ccC}}(\mathrm{D})}(\Psi(\mathrm{F}^{\star}\,\mathrm{X}),\mathrm{Y})\\
&\cong&\mathrm{Hom}_{\mathrm{inj}_{\underline{\ccC}}(\mathrm{D})}(\mathrm{F}^{\star}\,\mathrm{X},\Phi(\mathrm{Y}))\\
&\cong&\mathrm{Hom}_{\mathrm{inj}_{\underline{\ccC}}(\mathrm{D})}(\mathrm{X},\mathrm{F}\,\Phi(\mathrm{Y})).
\end{array}
\end{displaymath}
This shows that $\mathrm{F}\Phi(\mathrm{Y})$ is isomorphic to $\Phi(\mathrm{F}\,\mathrm{Y})$ and, by finite-dimensionality of morphisms spaces, the monic right map in \eqref{diag59} is, in fact, and isomorphism. 
The claim follows.
\end{proof}

\subsection{An alternative description of $\Phi$}\label{snew1.455}

Let $\cC$ be a fiat $2$-category. Consider two coalgebra $1$-morphisms  
$\mathrm{C}=(\mathrm{C},\tilde{\Delta},\tilde{\varepsilon})$ and 
$\mathrm{D}=(\mathrm{D},\Delta,\varepsilon)$ 
in $\underline{\cC}$. Consider also a monomorphism $\iota:\mathrm{D}\to\mathrm{C}$ 
of coalgebra $1$-morphisms and let $(\mathrm{J},\kappa)$ be the cokernel of $\iota$.
Then $\mathrm{J}$ is a {\em coideal} of $\mathrm{C}$ in the sense that there
exist $2$-morphisms $\lambda:\mathrm{J}\to \mathrm{C}\mathrm{J}$ and
$\rho:\mathrm{J}\to \mathrm{J}\mathrm{C}$ such that the following
diagrams commute:
\begin{displaymath}
\xymatrix{
\mathrm{C}\ar[d]_{\kappa}\ar[r]^{\tilde{\Delta}} & \mathrm{C}\mathrm{C}\ar[d]^{1\kappa}\\
\mathrm{J}\ar[r]^{\lambda} & \mathrm{C}\mathrm{J}
}\qquad\text{ and }\qquad
\xymatrix{
\mathrm{C}\ar[d]_{\kappa}\ar[r]^{\tilde{\Delta}} & \mathrm{C}\mathrm{C}\ar[d]^{\kappa 1}\\
\mathrm{J}\ar[r]^{\rho} & \mathrm{J}\mathrm{C}.
}
\end{displaymath}
Note that we have the exact sequence
\begin{displaymath}
\xymatrix{
0\ar[r]&\mathrm{D}\ar[r]^{\iota}&\mathrm{C}\ar[r]^{\kappa}&\mathrm{J}\ar[r]&0.
}
\end{displaymath}

Let $(\mathrm{M},\delta_{\mathrm{M}})$ be a $\mathrm{D}$-comodule $1$-morphism and
$(\mathrm{N},\delta_{\mathrm{N}})$ be a $\mathrm{C}$-comodule $1$-morphism.
Let $\varphi:\mathrm{M}\to \mathrm{N}$ be a morphism of  $\mathrm{C}$-comodule $1$-morphisms from
$\Psi(\mathrm{M})$ to $\mathrm{N}$. Consider the following diagram:
\begin{equation}\label{diag72}
\xymatrix{
&\mathrm{M}\ar[dd]_{\varphi}\ar[rr]^{\delta_{\mathrm{M}}}\ar@/_2pc/@{.>}[dddl]_{\xi_0}&&
\mathrm{M}\mathrm{D}\ar@{^{(}->}[rr]^{1_\mathrm{M}\iota}
\ar@/_2pc/@{.>}[dddddl]_>>>>>>>>>>>>>>>>>>>>>>{\xi_0 1_{\mathrm{D}}}
\ar[dddd]^<<<<<<<<<<<<<{\varphi 1_{\mathrm{D}}}&&
\mathrm{M}\mathrm{C}\ar[rr]^{1_\mathrm{M}\kappa}\ar[dd]^{\varphi 1_{\mathrm{C}}}&&
\mathrm{M}\mathrm{J}\ar[dd]^{\varphi 1_{\mathrm{J}}}\\
&&\mathrm{N}/\mathrm{K}_0&&&&&\\
&\mathrm{N}\ar[rrrr]^>>>>>>>>>{\delta_{\mathrm{N}}}\ar@{-->>}[ru]_{\pi_0}&&&&
\mathrm{N}\mathrm{C}\ar[rr]^{1_\mathrm{N} \kappa}&&\mathrm{N}\mathrm{J}\\
\mathrm{K}_0\ar@{^{(}-->}[ru]_{\tau_0}\ar@{.>}[rrrd]_<<<<<{\zeta_0}&&&&\mathrm{N}/\mathrm{K}_0\mathrm{D}&&&\\
&&&\mathrm{N}\mathrm{D}\ar@{-->}[ru]^>>>>>{\pi_0 1_{\mathrm{D}}}
\ar@/_2pc/@{_{(}->}[rruu]_{1_\mathrm{N} \iota}&&&&\\
&&\mathrm{K}_0\mathrm{D}\ar@{^{(}-->}[ru]_{\tau_0 1_{\mathrm{D}}}&&&&&
}
\end{equation}
The solid part of this diagram commutes by our assumptions. We define $(\mathrm{K}_0,\tau_0)$ as the kernel of
$(1_\mathrm{N}\kappa)\circ_{\mathrm{v}} \delta_{\mathrm{N}}$ with 
$(\mathrm{N}/\mathrm{K}_0,\pi_0)$ being the cokernel of $\tau_0$.
This gives the dashed parts of the diagram.
As the solid part commutes, the map from $\mathrm{M}$ to 
$\mathrm{N}\mathrm{J}$ is zero and hence factors through $\mathrm{K}_0$
by the universal property of kernels.
Next we note that the solid sequence
\begin{displaymath}
\mathrm{N}\mathrm{D}\hookrightarrow \mathrm{N}\mathrm{C}\longrightarrow \mathrm{N}\mathrm{J} 
\end{displaymath}
is left exact. Since the map from $\mathrm{K}_0$ to $\mathrm{N}\mathrm{J}$ is zero by definition of 
$\mathrm{K}_0$, this map factors
through $\mathrm{N}\mathrm{D}$ by the universal property of kernels. 
This  gives the dotted part of the diagram and the whole diagram commutes.

We would like to construct a $\mathrm{D}$-subcomodule $1$-morphism of $\mathrm{N}$. However, note that $\mathrm{K}_0$ is not yet a 
$\mathrm{D}$-comodule $1$-morphism as we have a map from $\mathrm{K}_0$ to $\mathrm{N}\mathrm{D}$, while we need a map from
$\mathrm{K}_0$ to $\mathrm{K}_0\mathrm{D}$. Therefore we need to do more.

Let $(\mathrm{K}_1,\tau_1)$ be the kernel of $(\pi_01_{\mathrm{D}})\circ_{\mathrm{v}} \zeta_0$.
By commutativity of \eqref{diag72}, the map from $\mathrm{M}$ to $\mathrm{N}/\mathrm{K}_0\mathrm{D}$ factors through 
$\mathrm{K}_0\mathrm{D}$ and hence is zero. Therefore $\xi_0$ factors via $\mathrm{K}_1$ through some map
$\xi_1:\mathrm{M}\to \mathrm{K}_1$. Moreover, as the sequence
\begin{displaymath}
\mathrm{K}_0\mathrm{D}\hookrightarrow \mathrm{N}\mathrm{D}\longrightarrow \mathrm{N}/\mathrm{K}_0\mathrm{D} 
\end{displaymath}
is left exact, the map from $\mathrm{K}_1$ to $\mathrm{N}\mathrm{D}$ factors through some $\zeta_1:\mathrm{K}_1\to \mathrm{K}_0\mathrm{D}$.
Now we can continue this and construct a sequence of monomorphisms
\begin{displaymath}
\dots \mathrm{K}_2 \overset{\tau_2}{\hookrightarrow}  \mathrm{K}_1\overset{\tau_1}{\hookrightarrow} \mathrm{K}_0, 
\end{displaymath}
with corresponding cokernels $(\mathrm{K}_i/\mathrm{K}_{i+1},\pi_i)$ together with factorization
maps 
\begin{displaymath}
\xi_i:\mathrm{M}\to \mathrm{K}_i\qquad\text{ and }\qquad\zeta_i:\mathrm{K}_i\to \mathrm{K}_{i-1}\mathrm{D} 
\end{displaymath}
such that the obtained infinite
diagram commutes.

As $\cC$ is finitary, this sequence must stabilize, that is $\mathrm{K}_i=\mathrm{K}_{i+1}$, for some
$i$. In this case $\mathrm{K}_i$ becomes a $\mathrm{D}$-comodule $1$-morphism by construction. We can now set 
$\tilde{\Phi}(\mathrm{N}):=\mathrm{K}_i$ and by functoriality of the above construction this defines a functor from
$\mathrm{comod}_{\underline{\ccC}}(\mathrm{C})$ to  $\mathrm{comod}_{\underline{\ccC}}(\mathrm{D})$.
The construction also automatically gives a functorial isomorphism
\begin{displaymath}
\mathrm{Hom}_{\mathrm{comod}_{\underline{\ccC}}(\mathrm{C})}(\Psi(\mathrm{M}),\mathrm{N})\cong
\mathrm{Hom}_{\mathrm{comod}_{\underline{\ccC}}(\mathrm{D})}(\mathrm{M},\tilde{\Phi}(\mathrm{N}))
\end{displaymath}
making $(\Psi,\tilde{\Phi})$ into an adjoint pair of functors.
By uniqueness of adjoint functors, the functor $\tilde{\Phi}$ defined in this
subsection is isomorphic to the functor ${}_-\square_{\mathrm{C}}\mathrm{D}$ from Subsection~\ref{snew1.45}.

\subsection{Cosimple coalgebra $1$-morphisms and simple transitive $2$-representations}\label{snew1.5}

From the above, we deduce a useful corollary strengthening \cite[Corollary~11]{MMMT}.

\begin{corollary}\label{simple}
Let $\mathbf{M}$ be a transitive $2$-representation and $X\in \mathbf{M}(\mathtt{i})$. Consider the coalgebra $1$-morphism $\mathrm{C}$ associated to $\mathbf{M}$ and $X$. Then $\mathrm{C}$ is cosimple if and only if $\mathbf{M}$ is simple transitive.
\end{corollary}

\begin{proof}
Assume $\mathrm{C}$ is cosimple and $\mathbf{M}$ is not simple transitive.
Let $\mathbf{N}$ be the simple transitive quotient of $\mathbf{M}$.
Then the natural projection $\mathbf{M}\tto \mathbf{N}$ is not an equivalence.
Therefore the homomorphism between the coalgebra $1$-morphisms $\mathrm{C}$ and $\mathrm{D}$,
for $\mathbf{M}$ and $\mathbf{N}$, respectively, given by  Proposition \ref{coalgebramorphism}
is not an isomorphism. Therefore $\mathrm{C}$ is not cosimple, a contradiction.

Assume now that $\mathrm{C}$ is not cosimple, then we have a monomorphism
\begin{equation}\label{daig51}
\iota:\mathrm{D}\hookrightarrow\mathrm{C} 
\end{equation}
for some coalgebra $1$-morphism $\mathrm{D}$ such that $\iota$ is not an isomorphism. 
Consider the functor ${}_-\square_{\mathrm{C}}\mathrm{D}$ from Subsection~\ref{snew1.45}.  
As $\iota$ is not an isomorphism and, by Corollary \ref{cormarco5}, ${}_-\square_{\mathrm{C}}\mathrm{D}$ is full when restricted
to injective comodule $1$-morphisms, we obtain that $\mathrm{inj}_{\underline{\ccC}}(\mathrm{D})$
is a proper quotient of $\mathrm{inj}_{\underline{\ccC}}(\mathrm{C})$, and hence  
$\mathrm{inj}_{\underline{\ccC}}(\mathrm{C})$ is not simple transitive.
The claim follows.
\end{proof}
  
\section{Centralizers in fiat $2$-categories}\label{s25}

\subsection{A general observation}\label{s2.5}

The following statement is similar in spirit to \cite[Corollary~12]{MMMT}.

\begin{proposition}\label{prop1}
Let $\cC$ be a fiat $2$-category and $\cA$ a $2$-subcategory of $\cC$. We assume that
\begin{enumerate}[$($a$)$]
\item\label{prop1.1} on the $2$-level, $\cA$ is $2$-full;
\item\label{prop1.2} on the $1$-level, $\cA$ is isomorphism closed and fully additively closed 
(in the sense of both being additively closed and idempotent complete);
\item\label{prop1.3} $\cA$ is stable under the weak involution $\star$ of $\cC$.
\end{enumerate}
Then every simple transitive $2$-representation of $\cA$ is equivalent to a simple transitive 
subquotient of the restriction, from $\cC$ to $\cA$, of a simple transitive $2$-representation of $\cC$.
\end{proposition}

\begin{proof}
Let $\mathbf{M}$ be a simple transitive $2$-representation of $\cA$.  
Note that the $2$-category $\cA$ is fiat due to assumption~\eqref{prop1.3}.
Therefore, by \cite[Theorem~9]{MMMT},
there is a coalgebra $1$-morphism $\mathrm{C}\in\underline{\cA}$ such that $\mathbf{M}$ is equivalent to
$\mathrm{inj}_{\underline{\ccA}}(\mathrm{C})$. Thanks to assumptions~\eqref{prop1.1} and \eqref{prop1.2},
the $2$-category $\underline{\cA}$ is, naturally, a full subcategory of  $\underline{\cC}$. 
This implies that $\mathrm{C}$ is also a coalgebra $1$-morphism in  $\underline{\cC}$
and $\mathrm{inj}_{\underline{\ccA}}(\mathrm{C})$ is, 
naturally, a full subcategory of $\mathrm{inj}_{\underline{\ccC}}(\mathrm{C})$.
The latter is a finitary $2$-representation of $\cC$. It has a unique simple transitive
subquotient $\mathbf{N}$ which contains $\mathrm{inj}_{\underline{\ccA}}(\mathrm{C})$
and in which the objects of $\mathrm{inj}_{\underline{\ccA}}(\mathrm{C})$ are non-zero.
Therefore $\mathbf{M}$ is equivalent to a simple transitive 
subquotient of the restriction of $\mathbf{N}$ from $\cC$ to $\cA$.
\end{proof}

\subsection{The special case of an $\mathcal{H}$-cell}\label{s2.6}

In this subsection, we strengthen the claim of Proposition~\ref{prop1} for a very special $2$-subcategory $\cA$ of $\cC$.
We assume that $\cC$ is a fiat $2$-category with a unique maximal two-sided cell $\mathcal{J}$ with respect to the two-sided order. By \cite[Corollary~19]{KM}, the two-sided cell $\mathcal{J}$ is
regular. Let $\mathcal{L}$ be a left cell in $\mathcal{J}$ and 
$\mathrm{G}$ a Duflo involution in $\mathcal{L}$. Finally, set  $\mathcal{H}:=\mathcal{L}\cap \mathcal{L}^{\star}$ 
(here $\mathcal{L}^{\star}$ can be alternatively described as the right cell containing $\mathrm{G}$). 
Let $\cA=\cA_{\mathcal{H}}$ be the $2$-full $2$-subcategory of $\cC$ generated by all $1$-morphisms in $\mathcal{H}$
together with the (unique) relevant identity morphism, that is, the identity $1$-morphism $\mathbbm{1}_{\mathtt{i}}$, where $\mathtt{i}$ is the domain of all $1$-morphisms in $\mathcal{L}$. Then $\cA$ is fiat and its indecomposable 
$1$-morphisms are $\mathbbm{1}_{\mathtt{i}}$ and those in $\mathcal{H}$. Moreover,
$\mathcal{H}$ is a two-sided cell in $\cA$.

\begin{proposition}\label{prop1n}
In the above setup, let $\mathbf{M}$ be a simple transitive $2$-rep\-re\-sen\-ta\-ti\-on of $\cC$ 
with apex $\mathcal{J}$. Then the restriction of $\mathbf{M}$ to $\cA$ contains a unique simple transitive
subquotient with apex $\mathcal{H}$.
\end{proposition}

\begin{proof}
Let $\mathrm{F}$ be a multiplicity free sum of representatives of isomorphism classes of indecomposable $1$-morphisms in $\mathcal{H}$.
Note that 
\begin{displaymath}
\mathcal{H}^{\star}=(\mathcal{L}\cap \mathcal{L}^{\star})^{\star}= \mathcal{L}\cap \mathcal{L}^{\star}=\mathcal{H},
\end{displaymath}
which implies that $\mathrm{F}^{\star}=\mathrm{F}$.
 
Consider $\mathbf{M}$ as a $2$-representation of $\cA$ by restriction. Let $\mathbf{N}$ be the $2$-subrepresentation 
given by the action of $\cA$ on the additive closure of all objects of the form $\mathrm{F}\,M$, where
$M\in \mathbf{M}(\mathtt{i})$. To prove our proposition, it remains to show that $\mathbf{N}$ is transitive.

First, we observe that  $\mathrm{add}(\mathrm{F}\,M)=\mathrm{add}(\mathrm{F}\,M')$, for all 
$M,M'\in \mathbf{M}(\mathtt{i})$ such that $\mathrm{F}\,M\neq 0$ and $\mathrm{F}\,M'\neq 0$.
Indeed, as $\mathrm{F}\,M'\neq 0$ and $\mathbf{M}$ is a transitive $2$-representation of 
$\cC$, there is a $1$-morphism $\mathrm{H}$ in $\cC$ such that $M$ belongs to 
$\mathrm{add}(\mathrm{H}\mathrm{F}\,M')$. This means that 
$\mathrm{add}(\mathrm{F}\,M)\subset \mathrm{add}(\mathrm{F}\mathrm{H}\mathrm{F}\,M')$.
However, $\mathrm{F}\mathrm{H}\mathrm{F}$ is in $\mathrm{add}(\mathrm{F})$ as $\mathcal{J}$ is regular and 
$\mathcal{H}:=\mathcal{L}\cap \mathcal{L}^{\star}$. Therefore 
$\mathrm{add}(\mathrm{F}\,M)\subset \mathrm{add}(\mathrm{F}\,M')$.
The opposite inclusion follows by symmetry.

Now we argue that the claim of the previous paragraph implies transitivity of $\mathbf{N}$.
From the previous paragraph, we have that $\mathbf{N}(\mathtt{i})$ coincides with
$\mathrm{add}(\mathrm{F}\,M)$, for any $M\in \mathbf{M}(\mathtt{i})$ such that $\mathrm{F}\,M\neq 0$.
Fix $M$ such that $\mathrm{F}\,M\neq 0$ and let $N\in \mathbf{N}(\mathtt{i})$. Then
\begin{displaymath}
0\neq \mathrm{Hom}_{\mathbf{M}(\mathtt{i})}(N,\mathrm{F}\,M)\cong
\mathrm{Hom}_{\mathbf{M}(\mathtt{i})}(\mathrm{F}\,N,M),
\end{displaymath}
as $\mathrm{F}=\mathrm{F}^{\star}$, which implies $\mathrm{F}\,N\neq 0$. By the previous paragraph,
$\mathrm{add}(\mathrm{F}\,N)=\mathbf{N}(\mathtt{i})$ which establishes transitivity of $\mathbf{N}$.
\end{proof}

\subsection{The main result}\label{s2.7}
Here we continue to work in the setup of Subsection~\ref{s2.6}.

Proposition~\ref{prop1n} defines a map, which we call $\Theta$, from the set of equivalence 
classes of simple transitive  $2$-representations of $\cC$ with apex $\mathcal{J}$ to the set 
of equivalence classes of simple transitive $2$-representations of $\cA$ with apex $\mathcal{H}$. 
Explicitly,  we take a simple transitive $2$-representation $\mathbf{M}$ of $\cC$ 
and define $\Theta(\mathbf{M})$  to be the simple transitive quotient of the 
$2$-subrepresentation of the restriction of $\mathbf{M}$ to $\cA$, given by the 
action of $\cA$ on the additive closure of all objects of the 
form $\mathrm{F}\,M$, for $M\in \mathbf{M}(\mathtt{i})$. 

Proposition~\ref{prop1} defines a map
$\Omega$ in the opposite direction as follows:
\begin{itemize}
\item take a simple transitive $2$-representation $\mathbf{N}$ 
of $\cA$ with apex $\mathcal{H}$;
\item let $X$ be a multiplicity free direct sum of representatives of isomorphism
classes of indecomposable objects in $\mathbf{N}(\mathtt{i})$;
\item let $\mathrm{D}_X$ be the unique (up to isomorphism) coalgebra $1$-morphism 
in $\underline{\cA}$ constructed in \cite[Lemma~5]{MMMT} such 
that $\mathbf{N}$ is equivalent to $\mathrm{inj}_{\underline{\ccA}}(\mathrm{D}_X)$;
\item set $\Omega(\mathbf{N})$ to be the unique simple transitive quotient of 
$\mathrm{inj}_{\underline{\ccC}}(\mathrm{D}_X)$ which contains $\mathrm{inj}_{\underline{\ccA}}(\mathrm{D}_X)$.
\end{itemize}

\begin{theorem}\label{thmmain}
The maps $\Theta$ and $\Omega$ constructed above provide mutually inverse bijections
$$\Theta\colon\left\{\begin{array}{cc}\mbox{equivalence classes of simple}\\ \mbox{ transitive
$2$-representations }\\ \mbox{ of $\cC$ with apex $\mathcal{J}$} 
\end{array}\right\} \leftrightarrow 
\left\{\begin{array}{cc}
\mbox{equivalence classes of simple } \\ \mbox{ transitive
$2$-representations } \\ \mbox{of $\cA$ with apex $\mathcal{H}$}
\end{array}\right\}\colon \Omega.$$
\end{theorem}

To prove Theorem~\ref{thmmain}, we will need the following lemma.

\begin{lemma}\label{lemnewnew01}
Assume that $\cC$ is $\mathcal{J}$-simple.
Any non-zero $2$-ideal in $\underline{\cC}$ contains 
a non-zero $2$-ideal of $\cC$, in particular, it contains 
the identity $2$-morphisms for all $1$-morphisms in $\mathcal{J}$.
\end{lemma}

\begin{proof}
Let $\mathrm{H}$ and $\mathrm{K}$ be two $1$-morphisms in $\underline{\cC}$
and $\alpha:\mathrm{H}\to\mathrm{K}$ a non-zero $2$-morphism contained in our non-zero $2$-ideal
in $\underline{\cC}$. Let $\mathrm{G}$ be a multiplicity-free sum of representatives of
isomorphism classes of all indecomposable $1$-morphisms in $\mathcal{J}$. Consider the $2$-morphism 
\begin{displaymath}
\beta:=1_{\mathrm{G}}\alpha1_{\mathrm{G}}: \mathrm{G}\mathrm{H}\mathrm{G}\to\mathrm{G}\mathrm{K}\mathrm{G},
\end{displaymath}
which is then also in the $2$-ideal.

We claim that $\beta$ is non-zero and that it induces an isomorphism between
some summands of $\mathrm{G}\mathrm{H}\mathrm{G}$ and $\mathrm{G}\mathrm{K}\mathrm{G}$, moreover,
these summands belong both to $\cC$ and to $\mathcal{J}$, up to isomorphism.

Let $\mathcal{L}$ be a left cell in $\mathcal{J}$. Consider the corresponding cell
$2$-representation $\mathbf{C}_{\mathcal{L}}$ of $\cC$. As $\cC$ is $\mathcal{J}$-simple,
this $2$-representation is faithful and its abelianization 
$\underline{\mathbf{C}_{\mathcal{L}}}$ is a faithful $2$-rep\-re\-sen\-tation
of $\underline{\cC}$ due to faithfulness of representability of functors by bimodules. 
By \cite[Theorem~2]{KMMZ}, in this $2$-representation,
the $1$-morphism $\mathrm{G}$ 
is represented by a projective functor. As $\mathbf{C}_{\mathcal{L}}$ is transitive
and $\cC$ is fiat, this projective functor is an additive generator
of the category of all projective endofunctors of $\coprod_{\mathtt{i}\in \ccC}\mathbf{C}_{\mathcal{L}}(\mathtt{i})$. 

Our first claim is that, in the extension of $\underline{\mathbf{C}_{\mathcal{L}}}$ to $\underline{\cC}$, both $\mathrm{G}\mathrm{H}\mathrm{G}$ and $\mathrm{G}\mathrm{K}\mathrm{G}$ act by
non-zero projective functors. We prove the claim for $\mathrm{G}\mathrm{H}\mathrm{G}$;
for $\mathrm{G}\mathrm{K}\mathrm{G}$ the arguments are similar.
Assume $\mathrm{H}$ is given by a diagram 
$\mathrm{H}_1\overset{\gamma}{\longrightarrow}\mathrm{H}_2$,
where $\mathrm{H}_1$ and $\mathrm{H}_2$ are now $1$-morphisms in $\cC$.
Then $\mathrm{G}\mathrm{H}\mathrm{G}$ is given by the diagram 
$\mathrm{G}\mathrm{H}_1\mathrm{G}\overset{1_{\mathrm{G}}\gamma1_{\mathrm{G}}}{\longrightarrow}
\mathrm{G}\mathrm{H}_2\mathrm{G}$. Here any indecomposable summand of both
$\mathrm{G}\mathrm{H}_1\mathrm{G}$ and $\mathrm{G}\mathrm{H}_2\mathrm{G}$
is in $\mathcal{J}$.

We claim that there are decompositions
\begin{equation}\label{eqdecomp}
\mathrm{G}\mathrm{H}_1\mathrm{G}\cong  \mathrm{X}_1\oplus \mathrm{X}_2\quad\text{ and }\quad
\mathrm{G}\mathrm{H}_2\mathrm{G} \cong \mathrm{Y}_1\oplus \mathrm{Y}_2
\end{equation}
such that $1_{\mathrm{G}}\gamma1_{\mathrm{G}}$ annihilates $\mathrm{X}_1$
and induces an isomorphism between $\mathrm{X}_2$ and $\mathrm{Y}_1$.
Let $B$ denote the underlying algebra of $\mathbf{C}_{\mathcal{L}}$, that is, $\coprod_{\mathtt{i}\in \ccC}\mathbf{C}_{\mathcal{L}}(\mathtt{i})$ is equivalent to $B$-$\mathrm{proj}$.
Assume, for simplicity, that $\mathrm{G}$ is represented by 
the $B$-$B$-bimodule $Q_1\otimes_{\Bbbk}P_1$, where $Q_1$ is a left projective $B$-module
and $P_1$ is a right projective $B$-module.
Assume, moreover, that $\mathrm{H}_1$ is similarly represented by 
the $B$-$B$-bimodule $Q_2\otimes_{\Bbbk}P_2$ and that $\mathrm{H}_2$ is represented by 
the $B$-$B$-bimodule $Q_3\otimes_{\Bbbk}P_3$. Note that, in general, $1$-morphisms in $\cC$
will be represented by sums of bimodules of such form but, as our arguments extend easily
to this general case, we will save the significant notational and technical
difficulties involved. Then  $1_{\mathrm{G}}\gamma1_{\mathrm{G}}$ corresponds to a homomorphism of
$B$-$B$-bimodules
\begin{displaymath}
Q_1\otimes_{\Bbbk}P_1\otimes_B Q_2\otimes_{\Bbbk}P_2\otimes_B  Q_1\otimes_{\Bbbk}P_1\to
Q_1\otimes_{\Bbbk}P_1\otimes_B Q_3\otimes_{\Bbbk}P_3\otimes_B  Q_1\otimes_{\Bbbk}P_1.
\end{displaymath}
Both functors are isomorphic to direct sums of copies of $Q_1\otimes_{\Bbbk}P_1$
with multiplicity vector spaces given by $P_1\otimes_B Q_2\otimes_{\Bbbk}P_2\otimes_B  Q_1$
and $P_1\otimes_B Q_3\otimes_{\Bbbk}P_3\otimes_B  Q_1$, respectively. The morphism
$1_{\mathrm{G}}\gamma1_{\mathrm{G}}$ defines a linear map between these multiplicity spaces.

Now, as the bimodule representing $\mathrm{G}$ is an additive generator for projective $B$-$B$-bimodules,
it follows that both multiplicity spaces are non-zero and the morphism defined by
$1_{\mathrm{G}}\gamma1_{\mathrm{G}}$ is non-zero. By choosing appropriate bases in these
multiplicity spaces, we obtain the decomposition \eqref{eqdecomp}.
The definition of injective abelianization implies that 
$\mathrm{G}\mathrm{H}_1\mathrm{G}\overset{1_{\mathrm{G}}\gamma1_{\mathrm{G}}}{\longrightarrow}
\mathrm{G}\mathrm{H}_2\mathrm{G}$ is isomorphic to $\mathrm{X}_1$.
The latter is in $\cC$ and each of its indecomposable summands is in $\mathcal{J}$.

Now we claim that  there are decompositions 
$\mathrm{G}\mathrm{H}\mathrm{G}\cong \mathrm{X}_1\oplus \mathrm{X}_2$ and 
$\mathrm{G}\mathrm{K}\mathrm{G}\cong \mathrm{Y}_1\oplus \mathrm{Y}_2$ such that
$\beta$ annihilates $\mathrm{X}_1$ and induces an isomorphism between $\mathrm{X}_2$ and
$\mathrm{Y}_1$. This is proved by the same arguments as above, which also show that, since the bimodule representing $\mathrm{G}$ is an additive generator for projective 
$B$-$B$-bimodules, $\beta$ is non-zero. This implies that $\mathrm{X}_2$ is non-zero
and the identity $1$-morphism on $\mathrm{X}_2$ belongs to our $2$-ideal.

Thus, any non-zero $2$-ideal in $\underline{\cC}$ contains a non-zero $2$-ideal of $\cC$.
The rest of the claim follows from $\mathcal{J}$-simplicity of $\cC$.
\end{proof}

\begin{proof}[Proof of Theorem~\ref{thmmain}]
It follows immediately from the definitions that $\Theta\circ\Omega$ is the identity.

To prove that $\Omega\circ\Theta$ is the identity, 
let $\mathbf{M}$ be a simple transitive $2$-representation of $\cC$ 
with apex  $\mathcal{J}$. Without loss of generality we may assume that $\cC$ is $\mathcal{J}$-simple.
As before, we let $\mathrm{F}$ be a multiplicity free sum of all 
indecomposable $1$-morphisms in $\mathcal{H}$. Let $X\in \mathbf{M}(\mathtt{i})$ 
be a multiplicity free direct sum of representatives 
of isomorphism classes of all indecomposable objects which 
appear in $\mathrm{F}\, Y$, where $Y\in\mathbf{M}(\mathtt{i})$ is an additive generator
(note that $X\neq 0$  as $0\neq \mathbf{M}(\mathrm{F})$ and
$\mathrm{add}(\mathbf{M}(\mathrm{F})\mathbf{M}(\mathrm{F}))=\mathrm{add}(\mathbf{M}(\mathrm{F}))$). 
Note also that, due to Proposition~\ref{prop1n}, our choice of $X$ here fits with the choice 
of $X$ which was used for the definition of $\Omega$ above.
By  \cite[Theorem~9]{MMMT},  there exists a 
coalgebra $1$-morphism $\mathrm{C}_X$ in $\underline{\cC}$ such that $\mathbf{M}$ is equivalent to 
$\mathrm{inj}_{\underline{\ccC}}(\mathrm{C}_X)$ and such that, for all $\mathrm{H}$ in $\cC$, we have
\begin{displaymath}
\mathrm{Hom}_{\underline{\mathbf{M}}}(X,\mathrm{H}\,X) \cong \mathrm{Hom}_{\underline{\ccC}}(\mathrm{C}_X,\mathrm{H}). 
\end{displaymath} 
The coalgebra object $\mathrm{C}_X$ is cosimple by Corollary~\ref{simple}.

Consider the inclusion $2$-functor $\iota\colon \cA \hookrightarrow \cC$. 
This extends  to a $2$-functor $\underline{\cA} \hookrightarrow \underline{\cC}$, 
whose component functors are left exact and
fully faithful, and send injective objects to injective objects.  
We denote by $\tau$ its (component-wise) left adjoint (cf. \cite{Au,KhM}), 
implying we have adjunction morphisms 
$\varepsilon\colon\tau\iota \to \mathrm{Id}_{\underline{\ccA}}$ and 
$\eta\colon\mathrm{Id}_{\underline{\ccC}} \to \iota\tau$, the first being an isomorphism.
Using these, for any $\mathrm{G},\mathrm{H}\in \underline{\cC}$, we have a morphism
$$\beta_{\mathrm{G}\mathrm{H}}\colon\tau(\mathrm{G}\mathrm{H}) \to \tau(\iota\tau(\mathrm{G})\iota\tau(\mathrm{H})) = \tau \iota(\tau(\mathrm{G})\tau(\mathrm{H})) \to \tau(\mathrm{G})\tau(\mathrm{H})$$
given by the composition of $\tau(\eta_\mathrm{G}\circ_{\mathrm{h}}\eta_\mathrm{H})$ with $\varepsilon_{\tau(\mathrm{G})\tau(\mathrm{H})}$, using, in the middle, that $\iota$ is a $2$-functor,
see \cite{Ke} for details. Consequently, $\tau$ is an oplax $2$-functor. In particular, 
$\tau(\mathrm{C}_X)$ has the canonical structure of a coalgebra $1$-morphism (cf. \cite{Ke}). 
The counit of $\tau(\mathrm{C}_X)$ is the image of the counit of $\mathrm{C}_X$ under $\tau$
(note that $\tau$, as a quotient functor, preserves the identity $1$-morphism).
The comultiplication for $\tau(\mathrm{C}_X)$ is the image of the comultiplication of
$\mathrm{C}_X$ under $\tau$ composed with $\beta_{\mathrm{C}_X\mathrm{C}_X}$.

As $\iota$ is a $2$-functor, by a similar reasoning $\iota\tau(\mathrm{C}_X)$ has the canonical 
structure of a coalgebra $1$-morphism. Furthermore, from the definitions it follows that 
the evaluation of $\eta$ at $\mathrm{C}_X$ is a homomorphism of coalgebra $1$-morphisms. 
This homomorphism
induces a morphism of $2$-representations $\Psi\colon \mathrm{comod}_{\underline{\ccC}}(\mathrm{C}_X)\to \mathrm{comod}_{\underline{\ccC}}(\iota\tau(\mathrm{C}_X))$. 
Note that $\Psi$ is exact and faithful
by construction. 

Now we claim that $\tau$ is faithful. Indeed, if $\tau$ were not faithful, then, by Lemma~\ref{lemnewnew01},
$\tau$ would annihilate all identity $2$-morphisms for all $1$-morphisms in $\mathcal{J}$.
This contradicts the fact that $\tau$ is adjoint to $\iota$ and that 
$\iota$ is faithful.

The fact that $\tau$ is faithful implies that $\eta_{\mathrm{C}_X}$ is monic.
Indeed the adjunction axioms (evaluated at $\mathrm{C}_X$) yield that 
$\tau(\eta_{\mathrm{C}_X})$ is monic, meaning $\tau$ kills the kernel of $\eta_{\mathrm{C}_X}$.
As $\tau$ is faithful, the kernel of $\eta_{\mathrm{C}_X}$ must be zero.
Hence, by
Subsection~\ref{snew1.45}, we also have the adjoint functor 
$\Phi$ which maps $\mathrm{inj}_{\underline{\ccC}}(\iota\tau(\mathrm{C}_X))$
to $\mathrm{inj}_{\underline{\ccC}}(\mathrm{C}_X)$ which is full on injective objects.

Now, for $\mathrm{G}, \mathrm{H}\in \mathcal{L}$,  we compute
\begin{equation}\label{eqn127}
\begin{split}
\mathrm{Hom}_{\mathrm{inj}_{\underline{\ccC}}(\iota\tau(\mathrm{C}_X))}
(\mathrm{G}\,\iota\tau(\mathrm{C}_X),\mathrm{H}\,\iota\tau(\mathrm{C}_X))&\cong 
\mathrm{Hom}_{\mathrm{inj}_{\underline{\ccC}}(\iota\tau(\mathrm{C}_X))}
(\iota\tau(\mathrm{C}_X),\mathrm{G}^{\star}\mathrm{H}\,\iota\tau(\mathrm{C}_X)) \\
&\cong 
\mathrm{Hom}_{\mathrm{inj}_{\underline{\ccC}}(\iota\tau(\mathrm{C}_X))}
(\iota\tau(\mathrm{C}_X),\iota\mathrm{G}^{\star}\mathrm{H}\,\tau(\mathrm{C}_X)) \\
&\cong 
\mathrm{Hom}_{\mathrm{inj}_{\underline{\ccA}}(\tau(\mathrm{C}_X))}
(\tau(\mathrm{C}_X),\mathrm{G}^{\star}\mathrm{H}\,\tau(\mathrm{C}_X)) \\
&\cong 
\mathrm{Hom}_{\underline{\ccA}}
(\tau(\mathrm{C}_X),\mathrm{G}^{\star}\mathrm{H}). \\
\end{split}
\end{equation}
Here in the first line we use involutivity of $\star$; in the second line we use 
that $\mathrm{G}^*\mathrm{H}\in \mathcal{H}$, for any $\mathrm{G}, \mathrm{H}\in \mathcal{L}$,
and hence $\mathrm{G}^*\mathrm{H}$ commutes with $\iota$; in the third line we use that $\iota$ is full and faithful;
and in the fourth line we use \cite[Lemma~7]{MMMT}. 

Similarly, we also have 
\begin{equation*}
\begin{split}
\mathrm{Hom}_{\mathrm{inj}_{\underline{\ccC}}(\mathrm{C}_X)}
(\mathrm{G}\,\mathrm{C}_X,\mathrm{H}\,\mathrm{C}_X)&\cong 
\mathrm{Hom}_{\mathrm{inj}_{\underline{\ccC}}(\mathrm{C}_X)}
(\mathrm{C}_X,\mathrm{G}^{\star}\mathrm{H}\,\mathrm{C}_X) \\
&\cong 
\mathrm{Hom}_{\underline{\ccC}}
(\mathrm{C}_X,\mathrm{G}^{\star}\mathrm{H}). \\
\end{split}
\end{equation*}

Now we note that, using adjunction and the fact that $\mathrm{G}^{\star}\mathrm{H}\in\cA$, we have
\begin{displaymath}
 \mathrm{Hom}_{\underline{\ccC}}
(\mathrm{C}_X,\mathrm{G}^{\star}\mathrm{H})\cong
 \mathrm{Hom}_{\underline{\ccC}}
(\mathrm{C}_X,\iota(\mathrm{G}^{\star}\mathrm{H}))\cong
\mathrm{Hom}_{\underline{\ccA}}
(\tau(\mathrm{C}_X),\mathrm{G}^{\star}\mathrm{H}).
\end{displaymath}

Combining this, we obtain
$$\mathrm{Hom}_{\mathrm{inj}_{\underline{\ccC}}(\mathrm{C}_X)}
(\mathrm{G}\,\mathrm{C}_X,\mathrm{H}\,\mathrm{C}_X) \cong \mathrm{Hom}_{\mathrm{inj}_{\underline{\ccC}}(\iota\tau(\mathrm{C}_X))}
(\mathrm{G}\,\iota\tau(\mathrm{C}_X),\mathrm{H}\,\iota\tau(\mathrm{C}_X)),$$ for each $\mathrm{G}, \mathrm{H}\in \mathcal{L}$.
Together with the fact that $\displaystyle \mathrm{add}(\{\mathrm{G}\, X\,:\,\mathrm{G}\in\mathcal{L}\})=
\coprod_{\mathtt{i}\in\ccC}\mathbf{M}(\mathtt{i})$, this
implies that $\Phi$ is, in fact, an equivalence and $\mathrm{inj}_{\underline{\ccC}}(\iota\tau(\mathrm{C}_X))\cong \mathbf{M}$ is simple transitive. 

Consequently, the coalgebra $1$-morphism $\iota\tau(\mathrm{C}_X)$ is cosimple by Corollary \ref{simple}. 
As $\iota$ is left exact, this implies that
the coalgebra $1$-morphism 
$\tau(\mathrm{C}_X)$ in $\underline{\cA}$ is cosimple. 
Thus, again by Corollary~\ref{simple}, the $2$-representation
$\mathrm{inj}_{\underline{\ccA}}(\tau(\mathrm{C}_X))$ is simple transitive.

On the other hand, by construction, $\mathrm{inj}_{\underline{\ccA}}(\tau(\mathrm{C}_X))$, which is given by the additive closure of $\mathrm{F}\, \tau(\mathrm{C}_X)$, is equivalent to the transitive $2$-subrepresentation $\mathbf{N}$ defined in the proof of Proposition \ref{prop1n} of the restriction of $\mathbf{M} \cong \mathrm{inj}_{\underline{\ccC}}(\iota\tau(\mathrm{C}_X))$ to $\cA$. As $\mathrm{inj}_{\underline{\ccA}}(\tau(\mathrm{C}_X))$ is simple transitive, this proves
\begin{displaymath}
\mathrm{inj}_{\underline{\ccA}}(\tau(\mathrm{C}_X))\cong\Theta(\mathrm{inj}_{\underline{\ccC}}(\iota\tau(\mathrm{C}_X))) \cong \Theta(\mathbf{M}).
\end{displaymath}
Now by definition, $\Omega(\mathrm{inj}_{\underline{\ccA}}(\tau(\mathrm{C}_X)))$ is the simple transitive quotient of $\mathrm{inj}_{\underline{\ccC}}(\iota\tau(\mathrm{C}_X))$, which is already simple transitive. Thus $$\Omega(\Theta(\mathbf{M}))\cong \Omega(\mathrm{inj}_{\underline{\ccA}}(\tau(\mathrm{C}_X))\cong \mathrm{inj}_{\underline{\ccC}}(\iota\tau(\mathrm{C}_X))\cong \mathbf{M}.$$
This, finally, yields that
$\Omega\circ\Theta$ is the identity and the proof is complete.
\end{proof}

\begin{corollary}\label{cor753}
The following statements hold.

\begin{enumerate}[$($i$)$]
\item \label{cor753.1}
In the construction of 
$\Omega(\mathbf{N})$, already $\mathrm{inj}_{\underline{\ccC}}(\mathrm{D}_X)$ is simple transitive.
\item \label{cor753.2}
In the construction of  $\Theta(\mathbf{M})$, we do not need to take any quotients,
that is, $\Theta(\mathbf{M})$ is a $2$-subrepresentation.
\end{enumerate}
\end{corollary}

\begin{proof}
Claim~\eqref{cor753.1} follows from the part of the proof of Theorem~\ref{thmmain}
which shows that $\Psi$ is an equivalence between 
$\mathrm{inj}_{\underline{\ccC}}(\iota\tau(\mathrm{C}_X))$ 
and the original simple transitive $2$-representation $\mathbf{M}$.

Claim~\eqref{cor753.2} is proved by noting, as in the last paragraph of the proof of Theorem~\ref{thmmain}, that $\mathrm{inj}_{\underline{\ccA}}(\tau(\mathrm{C}_X))$ is equivalent to the $2$-subrepresentation $\mathbf{N}$  of the restriction of $\mathbf{M}$ to $\cA$ constructed in the proof of Proposition \ref{prop1n}, and is proved to be simple transitive in the second to last paragraph 
of the proof of Theorem~\ref{thmmain}.
\end{proof}

\subsection{Some consequences}\label{s2.8}

\begin{corollary}\label{corn1}
Let $\cC$ be as in Subsection~\ref{s2.7} and 
$\mathcal{L}_i$, where $i=1,2$, be two different left cells in $\mathcal{J}$.
Then, for  $\mathcal{H}_i=\mathcal{L}_i\cap\mathcal{L}_i^\star$, where $i=1,2$, there is
a natural bijection between the set of equivalence classes of simple transitive 
$2$-representations of $\cA_{\mathcal{H}_1}$ with apex $\mathcal{H}_1$ and  
the set of equivalence classes of simple transitive  $2$-representations of $\cA_{\mathcal{H}_2}$
with apex $\mathcal{H}_2$.
\end{corollary}

\begin{proof}
This follows by applying Theorem~\ref{thmmain} first to the pair $\cC$ and $\cA_{\mathcal{H}_1}$
and then to the pair $\cC$ and $\cA_{\mathcal{H}_2}$.
\end{proof}

\begin{corollary}\label{corn2}
Let $\cC$ be a fiat $2$-category and $\mathcal{J}$ a two-sided cell in $\cC$.
Assume that $\mathcal{J}$ contains a left cell $\mathcal{L}$ such that 
$|\mathcal{L}\cap\mathcal{L}^\star|=1$. Then $\mathcal{J}$ is strongly regular.
\end{corollary}

\begin{proof}
Due to \cite[Corollary~19]{KM}, we already know  that $\mathcal{J}$ is regular.
Without loss of generality we may assume that $\cC$ is as in Subsection~\ref{s2.6}.
Let $\mathcal{H}=\mathcal{L}\cap\mathcal{L}^\star$ and $\cA=\cA_{\mathcal{H}}$.
Then the two-sided cell $\mathcal{H}$ in $\cA$ is strongly regular and consists of one element. 
Therefore, due to \cite[Theorem~18]{MM5}
(in the version of
\cite[Theorem~33]{MM6}), $\cA$ has a unique, up to equivalence, simple
transitive $2$-representation with apex $\mathcal{H}$.

From Theorem~\ref{thmmain} we thus obtain that $\cC$ has a unique, up to equivalence, simple
transitive $2$-representation with apex $\mathcal{J}$. Consequently, for any two left cells
$\mathcal{L}_1$ and $\mathcal{L}_2$ in $\mathcal{J}$, the corresponding cell 
$2$-representations $\mathbf{C}_{\mathcal{L}_1}$ and $\mathbf{C}_{\mathcal{L}_2}$ of $\cC$
are equivalent. Let $\Psi:\mathbf{C}_{\mathcal{L}_1}\to \mathbf{C}_{\mathcal{L}_2}$
be an equivalence. Then $\Psi$ induces a bijection from $\mathcal{L}_1$ (which indexes isomorphism
classes of indecomposable objects in 
$\displaystyle\coprod_{\mathtt{i}\in\ccC}\mathbf{C}_{\mathcal{L}_1}(\mathtt{i})$) 
to $\mathcal{L}_2$ (which indexes isomorphism classes of indecomposable objects in 
$\displaystyle\coprod_{\mathtt{i}\in\ccC}\mathbf{C}_{\mathcal{L}_2}(\mathtt{i})$).

For $j=1,2$, let $L_{\mathrm{F}_j}$ be the simple object in 
$\displaystyle\coprod_{\mathtt{i}\in\ccC}\overline{\mathbf{C}_{\mathcal{L}_j}}(\mathtt{i})$ 
corresponding to some $\mathrm{F}_j\in \mathcal{L}_j$.
By \cite[Lemma~12]{MM1}, the annihilators of $L_{\mathrm{F}_1}$ and $L_{\mathrm{F}_2}$
in $\cC$ have a chance to coincide only if $\mathrm{F}_1$ and $\mathrm{F}_2$ belong
to the same right cell in $\mathcal{J}$. Therefore the bijection from 
$\mathcal{L}_1$ to $\mathcal{L}_2$ induced by $\Psi$ preserves right cells.
In other words, for any right cell $\mathcal{R}$ in $\mathcal{J}$, we have
\begin{equation}\label{eqeqn7}
|\mathcal{L}_1\cap\mathcal{R}|=|\mathcal{L}_2\cap\mathcal{R}|. 
\end{equation}

Using \eqref{eqeqn7}, from $|\mathcal{H}|=1$, we thus get $|\mathcal{L}'\cap \mathcal{L}^{\star}|=1$,
for any left cell $\mathcal{L}'$ in $\mathcal{J}$. Applying $\star$,  we get
$|\mathcal{R}\cap \mathcal{L}|=1$,
for any right cell $\mathcal{R}$ in $\mathcal{J}$. Now, using  \eqref{eqeqn7} once more,
it follows that $\mathcal{J}$ is strongly regular.
\end{proof}

\section{Application to Soergel bimodules, part I}\label{s3}

\subsection{Soergel bimodules and their small quotients}\label{s3.1}

Let $(W,S)$ be a finite Coxeter system. Further, let $V_{\mathbb{R}}$ be a reflection faithful $W$-module,
$V$ its complexification and $\mathtt{C}$ the corresponding {\em coinvariant algebra}. Let $\cS$ denote
the fiat $2$-category of Soergel $\mathtt{C}$-$\mathtt{C}$-bimodules associated with $(W,S)$ and $V$, 
see \cite{So1,So2,EW}. The unique object in $\cS$ is denoted $\mathtt{i}$.
For $w\in W$, we denote by $\theta_w$ the unique (up to isomorphism) indecomposable
Soergel bimodule which corresponds to $w$. Then the split Grothendieck ring of $\cS$ is isomorphic 
to $\mathbb{Z} [W]$ and this isomorphism sends the class of $\theta_w$ to the corresponding element
$\underline{H}_w$ of the Kazhdan-Lusztig basis in $\mathbb{Z} [W]$, see \cite{KL,EW}. Consequently,
the left, right and two-sided cells in $\cS$ are given by the corresponding left, right and two-sided
Kazhdan-Lusztig cells in $W$.

The $2$-category $\cS$ has a unique minimal (with respect to the two-sided order) 
two-sided cell, namely, the two-sided cell  $\mathcal{J}_e:=\{\theta_e\}$. In the set of the 
remaining two-sided cells there is again a unique minimal (with respect to the two-sided order)
two-sided cell which we call $\mathcal{J}$ for the rest of this section. 
This two-sided cell contains all $\theta_s$, where
$s\in S$. According to \cite[Proposition~4]{KMMZ}, elements of $\mathcal{J}$ can be characterized
in many different ways, for example, as those $\theta_w$, where $w\in W$, for which $w$
has a unique reduced expression.

In the set of all $2$-ideals of $\cS$ which do not contain any $1_{\theta_w}$, 
where $\theta_w\in \mathcal{J}$,
there is a unique ideal that is maximal with respect to inclusions. The corresponding quotient
of $\cS$ is denoted by ${{\cS}_{\hspace{-1mm}\mathrm{sm}}}$ and called the {\em small quotient} of $\cS$. 
We refer the reader to \cite[Section~3]{KMMZ} for details (note the change of notation 
from $\underline{\cS}$ to  ${{\cS}_{\hspace{-1mm}\mathrm{sm}}}$ to avoid
conflict with injective abelianization).

\subsection{Simple transitive $2$-representations of ${{\cS}_{\hspace{-1mm}\mathrm{sm}}}$}\label{s3.2}

The following statement, essentially, summarizes \cite[Theorem~6.1]{Zi1}, 
\cite[Theorem~1]{KMMZ} and \cite[Theorems~I,~II,~III]{MT}.
\vspace{7mm}

\begin{theorem}\label{thm2}
\hspace{2mm}

\begin{enumerate}[$($i$)$]
\item\label{thm2.1} If $(W,S)$ is not of type $I_2(n)$ with $n>4$ even, then every simple transitive 
$2$-representation of ${{\cS}_{\hspace{-1mm}\mathrm{sm}}}$ is equivalent to a cell $2$-representation.
\item\label{thm2.2} If $(W,S)$ is of type $I_2(n)$ with $n>4$ even and $n\neq 12,18,30$, then,
apart from cell $2$-representations, ${{\cS}_{\hspace{-1mm}\mathrm{sm}}}$ has exactly two extra equivalence classes of 
simple transitive  $2$-representations.
\item\label{thm2.3} If $(W,S)$ is of type $I_2(n)$ with $n= 12,18,30$, then,
apart from cell $2$-rep\-re\-sen\-ta\-ti\-ons, ${{\cS}_{\hspace{-1mm}\mathrm{sm}}}$ has at least four extra equivalence 
classes of simple transitive  $2$-rep\-re\-sen\-ta\-ti\-ons. Under the additional assumption of gradability,
these exhaust all simple transitive  $2$-representations.
\end{enumerate}
\end{theorem}

\subsection{The $2$-category ${\cQ}$}\label{s3.3}

Consider the $2$-category ${{\cS}_{\hspace{-1mm}\mathrm{sm}}}$ and fix $s\in S$. 
Denote by $\cQ$ the $2$-subcategory $\cA_{\mathcal{H}}$ of ${{\cS}_{\hspace{-1mm}\mathrm{sm}}}$, where 
$\mathcal{H}=\mathcal{L}_{s}\cap \mathcal{L}_{s}^{\star}$ with
$\mathcal{L}_{s}$ being the left cell containing $\theta_s$.
In the cases when $(W,S)$ is of type $I_2(4)$ or $I_2(5)$, the $2$-category $\cQ$ was studied in \cite{MaMa}
where  all simple transitive $2$-representations  of $\cQ$ were classified.

If $(W,S)$ is simply laced, then both two-sided cells of ${{\cS}_{\hspace{-1mm}\mathrm{sm}}}$, 
and hence also of $\cQ$,
are strongly regular, see \cite[Corollary~5]{KMMZ}. 
Therefore, in this case, the classification of simple transitive $2$-representations 
for both  ${{\cS}_{\hspace{-1mm}\mathrm{sm}}}$ and $\cQ$ follows directly from 
\cite[Theorem~18]{MM5} in the version of
\cite[Theorem~33]{MM6}.

In this section we combine Theorems~\ref{thmmain} and \ref{thm2} to classify simple transitive 
$2$-rep\-re\-sen\-ta\-ti\-ons of $\cQ$. 

\subsection{The case of $I_2(n)$ with $n$ odd}\label{s3.4}

The following result generalizes \cite[Theorem~5]{MaMa} where the case $n=5$ was considered. Our approach is 
rather different. We denote by $\mathcal{L}$ the left cell of $\cQ$ containing $\theta_s$.

\begin{theorem}\label{thm4}
If $(W,S)$ is of type  $I_2(n)$ with $n$ odd, then every simple transitive $2$-rep\-re\-sen\-tation of
$\cQ$ is equivalent to a cell $2$-representation.
\end{theorem}

\begin{proof}
By Theorem~\ref{thm2}, under our assumptions, every simple transitive $2$-representation of
${{\cS}_{\hspace{-1mm}\mathrm{sm}}}$ is a cell $2$-representation. The $2$-category ${{\cS}_{\hspace{-1mm}\mathrm{sm}}}$ has
three left cells: the left cell $\mathcal{L}_e$ consisting of $\theta_e$,
the left cell $\mathcal{L}_s$ containing $\theta_s$ and
the left cell $\mathcal{L}_t$ containing $\theta_t$, where $t\in S$ is the other simple
reflection.

The cell $2$-representation $\mathbf{C}_{\mathcal{L}_e}$ of ${{\cS}_{\hspace{-1mm}\mathrm{sm}}}$
is just an action of ${{\cS}_{\hspace{-1mm}\mathrm{sm}}}$ on complex 
vector spaces where $\theta_e$ acts as the identity and
all other $\theta_w$ acts as zero. Via the bijection in Theorem~\ref{thmmain}
(applied to the appropriate quotients of ${{\cS}_{\hspace{-1mm}\mathrm{sm}}}$ and $\cQ$)
this corresponds to the 
cell $2$-representation $\mathbf{C}_{\mathcal{L}_e}$ of $\cQ$.
It is also clear that the bijection in Theorem~\ref{thmmain} matches the 
cell $2$-representation $\mathbf{C}_{\mathcal{L}}$ of $\cQ$ with the
cell $2$-representation $\mathbf{C}_{\mathcal{L}_s}$ of ${{\cS}_{\hspace{-1mm}\mathrm{sm}}}$.

To complete the proof, it is now left to show that the $2$-representations 
$\mathbf{C}_{\mathcal{L}_s}$ and $\mathbf{C}_{\mathcal{L}_t}$ of 
${{\cS}_{\hspace{-1mm}\mathrm{sm}}}$ are equivalent. Indeed, consider the
abelianization $\overline{\mathbf{C}_{\mathcal{L}_t}}$.
In the category $\overline{\mathbf{C}_{\mathcal{L}_t}}(\mathtt{i})$,
consider the simple top $L_{w_0s}$ of the projective object 
$0\to\theta_{w_0s}$.
By the Yoneda Lemma, there is a unique $2$-natural transformation $\Psi$ 
from the principal  $2$-representation $\mathbf{P}_{\mathtt{i}}$ to 
$\overline{\mathbf{C}_{\mathcal{L}_t}}$ which sends
$\mathbbm{1}_{\mathtt{i}}$ to $L_{w_0s}$. The multiplication table
of ${{\cS}_{\hspace{-1mm}\mathrm{sm}}}$ is obtained inductively from the following formulae:
\begin{displaymath}
\theta_s \theta_w=
\begin{cases}
\theta_s, & w=e;\\
\theta_{st}, & w=t;\\
\theta_{tsw_0}, & w=sw_0;\\
\theta_{w}\oplus \theta_{w}, & \ell(sw)<\ell(w);\\
\theta_{sw}\oplus \theta_{tw}, & \text{else}.
\end{cases}\qquad
\theta_t \theta_w=
\begin{cases}
\theta_t, & w=e;\\
\theta_{ts}, & w=s;\\
\theta_{stw_0}, & w=tw_0;\\
\theta_{w}\oplus \theta_{w}, & \ell(tw)<\ell(w);\\
\theta_{tw}\oplus \theta_{sw}, & \text{else}.
\end{cases}
\end{displaymath}
We note the symmetry of the above formulae with respect to
swapping, for $w$, the elements $s$ and $w_0s$, the elements
$ts$ and $w_0ts$, and so on. 
From this and the explicit description of the endomorphism algebras of 
the underlying category $\mathbf{C}_{\mathcal{L}_t}(\mathtt{i})$
in \cite[Subsection~7.2]{KMMZ},
it follows easily that $\Psi$ maps 
$\theta_w\in\mathcal{L}_s$ to an indecomposable projective object
in $\mathbf{C}_{\mathcal{L}_t}(\mathtt{i})$.
By comparing the Cartan matrices of the endomorphism algebras of 
the underlying categories, we see this $\Psi$ induces an equivalence
between $\mathbf{C}_{\mathcal{L}_s}$ and $\mathbf{C}_{\mathcal{L}_t}$.
The claim of the theorem follows.
\end{proof}

\subsection{The case of $I_2(4)$}\label{s3.5}

For $\cX\in\{\cQ,{{\cS}_{\hspace{-1mm}\mathrm{sm}}}\}$ and $w\in\{e,s,t\}$ (here $w\neq t$ if $\cX=\cQ$),
we denote by $\mathcal{L}_w^{(\ccX)}$ the left cell of $\cX$ containing $\theta_w$.
The following result significantly simplifies the proof of  \cite[Theorem~12]{MaMa}. 

\begin{theorem}\label{thm5}
If $(W,S)$ is of type  $I_2(4)$, then $\cQ$
has exactly three equivalence classes of simple transitive $2$-rep\-re\-sen\-tations,
namely $\mathbf{C}_{\mathcal{L}_e^{(\ccQ)}}$, $\mathbf{C}_{\mathcal{L}_s^{(\ccQ)}}$
and $\Theta(\mathbf{C}_{\mathcal{L}_t^{({\ccS}_{\hspace{-1mm}\mathrm{sm}})}})$.
\end{theorem}

We note that the $2$-representation
$\Theta(\mathbf{C}_{\mathcal{L}_t^{({\ccS}_{\hspace{-1mm}\mathrm{sm}})}})$ is not equivalent to
a cell $2$-representation.

\begin{proof}
By \cite[Theorem~6.12]{Zi1}, up to equivalence, simple transitive $2$-representations of ${{\cS}_{\hspace{-1mm}\mathrm{sm}}}$ are
exactly $\mathbf{C}_{\mathcal{L}_e^{({\ccS}_{\hspace{-1mm}\mathrm{sm}})}}$,
$\mathbf{C}_{\mathcal{L}_s^{({\ccS}_{\hspace{-1mm}\mathrm{sm}})}}$ and
$\mathbf{C}_{\mathcal{L}_t^{({\ccS}_{\hspace{-1mm}\mathrm{sm}})}}$.
Now the claim follows from Theorem~\ref{thmmain}
similarly to the first part of the proof of Theorem~\ref{thm4}.
\end{proof}

\begin{remark}\label{remMaMa}
{\em As already mentioned, the argument above is a significant simplification of 
\cite[Theorem~12]{MaMa}. However, the original proof of \cite[Theorem~12]{MaMa} is crucially
used in \cite{KMMZ} and hence in the results below.
}
\end{remark}

\subsection{The case of $I_2(n)$ with $n\neq 12,18,30$ and $n>4$ even}\label{s3.7}

\begin{theorem}\label{thm6}
If $(W,S)$ is of type $I_2(n)$ with $n\neq 12,18,30$ and $n>4$ even, then $\cQ$
has exactly five equivalence classes of simple transitive $2$-rep\-re\-sen\-tations
(the images under $\Theta$ of the  five equivalence classes of simple transitive
$2$-representations of ${{\cS}_{\hspace{-1mm}\mathrm{sm}}}$).
\end{theorem}

\begin{proof}
Similarly to the above, this follows from Theorems~\ref{thmmain} and \ref{thm2}.
We just note that, for even $n$, all cell $2$-representations of 
${{\cS}_{\hspace{-1mm}\mathrm{sm}}}$ are inequivalent.
\end{proof}

\subsection{The cases of $I_2(12)$, $I_2(18)$ and $I_2(30)$}\label{s3.9}

\begin{proposition}\label{prop8}
If $(W,S)$ is of type $I_2(12)$, $I_2(18)$ or $I_2(30)$, then $\cQ$
has at least seven equivalence classes of simple transitive $2$-rep\-re\-sen\-tations
(the images under $\Theta$ of the seven equivalence classes of simple transitive
$2$-representations of ${{\cS}_{\hspace{-1mm}\mathrm{sm}}}$).
Under the additional assumption of gradability, these exhaust all 
simple transitive $2$-rep\-re\-sen\-tations of $\cQ$.
\end{proposition}

\begin{proof}
Similarly to the above, this follows from Theorems~\ref{thmmain} and \ref{thm2}.
\end{proof}

\subsection{The cases of rank higher than $2$}\label{s3.10}

\begin{theorem}\label{thm8-1}
If $(W,S)$ is of type $H_3$ or $H_4$, then every simple 
transitive $2$-rep\-re\-sen\-tation of
$\cQ$ is equivalent to a cell $2$-representation.
\end{theorem}

\begin{proof}
From \cite[Subsection~4.2]{KMMZ}, the indecomposable $1$-morphisms in 
$\cQ$ are $\theta_e$, $\theta_s$ and $\theta_{sts}$, where $s$ and $t$
are two simple reflections such that $(st)^5=e$. The general definition of
Soergel bimodules implies that the 
$\mathcal{H}$-simple quotient of $\cQ$ in types $H_3$ and $H_4$ is biequivalent
to the $\mathcal{H}$-simple quotient of the $2$-category $\cQ$ in type $I_2(5)$.
Therefore the claim follows from Theorem~\ref{thm4} (see also \cite[Theorem~5]{MaMa}).
\end{proof}

\begin{corollary}\label{cor8-11}
If $(W,S)$ is of type $H_3$ or $H_4$ and $\mathcal{L}_1$ and $\mathcal{L}_2$
are two left cells in $\mathcal{J}$, then the cell $2$-representations 
$\mathbf{C}_{\mathcal{L}_1}$ and $\mathbf{C}_{\mathcal{L}_2}$ of
${{\cS}_{\hspace{-1mm}\mathrm{sm}}}$ are equivalent.
\end{corollary}

\begin{proof}
This follows from Theorem~\ref{thm8-1} as $\mathcal{H}$ is the only non-identity
two-sided cell in $\cQ$.
\end{proof}

\begin{theorem}\label{thm8-2}
If $(W,S)$ is of type $F_4$ or $B_n$, for $n>2$, then $\cQ$ has three 
equivalence classes of simple 
transitive $2$-rep\-re\-sen\-tations, two of which are cell $2$-representations.
\end{theorem}

\begin{proof}
From \cite[Subsection~4.2]{KMMZ}, the indecomposable $1$-morphisms in 
$\cQ$ are $\theta_e$, $\theta_s$ and $\theta_{sts}$, where $s$ and $t$
are two simple reflections such that $(st)^4=e$. The general definition of
Soergel bimodules implies that the 
$\mathcal{H}$-simple quotient of $\cQ$ in types $F_4$ or $B_n$, for $n>2$, is biequivalent
to the $\mathcal{H}$-simple quotient of the $2$-category $\cQ$ in type $I_2(4)$.
Therefore the claim follows from Theorem~\ref{thm5} (see also \cite[Theorem~12]{MaMa}).
\end{proof}

\begin{corollary}\label{cor8-22}
If $(W,S)$ is of type $F_4$ or $B_n$, for $n>2$, then there are exactly two equivalence classes of 
cell $2$-representations of 
${{\cS}_{\hspace{-1mm}\mathrm{sm}}}$ with apex $\mathcal{J}$.
\end{corollary}

\begin{proof}
This follows from Theorems~\ref{thm2} and \ref{thm8-2} as in this case $\cQ$ has two
equivalence classes of simple 
transitive $2$-rep\-re\-sen\-tations with apex $\mathcal{H}$.
\end{proof}

\section{Application to Soergel bimodules, part II}\label{s4}

\subsection{Nice two-sided cells}\label{s4.1}

In this section, we let $(W,S)$ be a finite Coxeter system and $\cS$ the corresponding $2$-category 
of Soergel bimodules, see Subsection~\ref{s3.1} for details. Let $\mathcal{J}$ be a two-sided 
cell of $\cS$. We will say that $\mathcal{J}$ is {\em nice} provided that the following
conditions are satisfied.
\begin{itemize}
\item For any left cell $\mathcal{L}$ in $\mathcal{J}$ and any right cell  
$\mathcal{R}$ in $\mathcal{J}$, we have $|\mathcal{L}\cap\mathcal{R}|\leq 2$;
\item for any left cell $\mathcal{L}$ in $\mathcal{J}$ there are right cells 
$\mathcal{R}_1$ and $\mathcal{R}_2$ in $\mathcal{J}$ 
such that  we have $|\mathcal{L}\cap\mathcal{R}_1|=1$ and $|\mathcal{L}\cap\mathcal{R}_2|=2$;
\item $\mathcal{J}$ contains the longest element for some parabolic subgroup of $W$.
\end{itemize}

For example, from \cite[Subsection~4.2]{KMMZ} it follows that the two-sided cell of $\cS$ which
contains all simple reflections is nice if and only if $W$ is of type $F_4$ or $B_n$, for $n\geq 2$.

\begin{example}\label{exfc}
{\em  
Recall that, for a Coxeter group $(W,S)$, an element $w\in W$ is {\em fully commutative}
provided that any two reduced expressions of $w$ can be obtained from each other only using
commutativity of Coxeter generators, see \cite{St}. By \cite[Theorem~3.1.1]{GL}, in type $B_n$, 
the set of all fully commutative elements is a union of two-sided cells. Assume $W$ is of type
$B_n$ and let $S'\subset S$ be a non-empty subset such that any pair of elements in $S'$ commutes.
Let $w'_0$ be the product of all elements in $S'$ and $\mathcal{J}$ the two-sided cell containing $w'_0$.
Note that $w'_0$ is the longest element of the parabolic subgroup generated by $S'$.
With this in mind, the very last paragraph of \cite{GL}, just before the bibliography, 
implies that $\mathcal{J}$ is nice.
}
 \end{example}

\subsection{Simple transitive $2$-representations with nice apex}\label{s4.2}
The main result of this section is the following theorem.

\begin{theorem}\label{thm55}
Let $\mathbf{M}$ be a simple transitive $2$-representation of $\cS$ with a nice apex $\mathcal{J}$. 
Then $\mathbf{M}$ is equivalent to a cell $2$-representation. Moreover, there are exactly 
two equivalence classes of simple transitive $2$-representations of $\cS$ with apex $\mathcal{J}$.
\end{theorem}

\begin{proof}
Let $\cC$ be the minimal quotient of $\cS$
in which $1$-morphisms of $\mathcal{J}$ survive (the $\mathcal{J}$-simp\-le quotient in the
terminology of \cite{MM1}). Let $w'_0$ be an element of $\mathcal{J}$ which is the longest
element of some parabolic subgroup of $W$, say $W'$. Let $\mathcal{L}$ and $\mathcal{R}$
be the left and the right cells containing $w'_0$, respectively. Set
$\mathcal{H}:=\mathcal{L}\cap \mathcal{R}$. Let $\cA$ be the $2$-subcategory of $\cC$
generated by $\theta_e$ and the elements in $\mathcal{H}$.
 
Note that $\mathcal{R}=\mathcal{L}^{\star}$. Therefore, if we assume 
$|\mathcal{H}|=1$, from Corollary~\ref{corn2} we deduce that $\mathcal{J}$ is strongly regular.
This, however, contradicts our assumption that $\mathcal{J}$ is nice.
Therefore $|\mathcal{H}|=1$ is not possible.

The above shows that $|\mathcal{H}|>1$, that is $|\mathcal{H}|=2$ as $\mathcal{J}$ is nice.
Let $\theta_w$ be the element in $\mathcal{H}$ different from $\theta_{w'_0}$.
Let $\cS\cS$ be the $2$-category of {\em singular Soergel bimodules}, see \cite[Subsection~7.2]{MM1}. 
The objects of  $\cS\cS$ correspond to parabolic subgroups of $W$. If  $\mathtt{i}$ is the object
corresponding to the trivial subgroup, then the full $2$-subcategory of $\cS\cS$
with object $\mathtt{i}$
is naturally identified with $\cS$. We
let $\mathtt{j}$ be the object corresponding to $W'$. Then there are $1$-morphisms
$\theta^{\mathrm{on}}\in \cS\cS(\mathtt{i},\mathtt{j})$ and
$\theta^{\mathrm{out}}\in \cS\cS(\mathtt{j},\mathtt{i})$ such that 
\begin{displaymath}
\theta_{w'_0}\cong\theta^{\mathrm{out}}\theta^{\mathrm{on}}\quad\text{ and }\quad
\mathbbm{1}_{\mathtt{j}}^{\oplus\mathbf{l}(w'_0)}\cong\theta^{\mathrm{on}}\theta^{\mathrm{out}},
\end{displaymath}
where $\mathbf{l}$ denotes the length function. We denote by $\theta$ the $1$-morphism in
$\cS\cS(\mathtt{j},\mathtt{j})$ such that 
\begin{displaymath}
\theta_w\cong \theta^{\mathrm{out}}\theta \theta^{\mathrm{on}}.
\end{displaymath}
Let $\cS\cS_{\mathtt{j}}$ be the full $2$-subcategory of $\cS\cS$
with object $\mathtt{j}$. Let $\cA^\checkmark$ be the $2$-subcategory of 
$\cS\cS_{\mathtt{j}}$ given by $\mathbbm{1}_{\mathtt{j}}$
and $\theta$. Combining Theorem~\ref{thmmain} with \cite[Corollary~12]{MMMT}, we obtain that there
is a bijection between simple transitive $2$-representations of $\cS$ with apex $\mathcal{J}$
and simple transitive $2$-representations of $\cA^\checkmark$.

Next we claim that $\cS$ has at least two equivalence classes of cell $2$-representations with
apex $\mathcal{J}$. This is similar to the argument in the proof of 
Corollary~\ref{corn2}. As $\mathcal{J}$ is nice, 
there are left cells 
$\mathcal{L}_1$ and $\mathcal{L}_2$ in $\mathcal{J}$ 
such that  $|\mathcal{L}_1\cap\mathcal{R}|=1$ and $|\mathcal{L}_2\cap\mathcal{R}|=2$.
In the proof of Corollary~\ref{corn2}, we show that an equivalence between 
$\mathbf{C}_{\mathcal{L}_1}$ and $\mathbf{C}_{\mathcal{L}_2}$ would
induce a bijection between $\mathcal{L}_1\cap\mathcal{R}$ and $\mathcal{L}_2\cap\mathcal{R}$. 
Since these have different cardinalities in our
situation, we deduce that $\mathbf{C}_{\mathcal{L}_1}$
and $\mathbf{C}_{\mathcal{L}_2}$ are not equivalent.

To complete the proof of the proposition, it is now sufficient to show that 
$\cA^\checkmark$ has at most two equivalence classes of simple transitive $2$-representations.
From the previous paragraph, we know that it has at least two, moreover, one of them has
rank one and another one has rank two.
Note that $\mathbbm{1}_{\mathtt{j}}$ and $\theta$ belong to the same two-sided cell of $\cA^\checkmark$
for otherwise there would exist a $2$-representation of $\cA^\checkmark$ which annihilates
$\theta$ but not $\mathbbm{1}_{\mathtt{j}}$. This would give, via \cite[Corollary~12]{MMMT},
a $2$-representation of $\cC$ which annihilates $\theta_{w}$ but not $\theta_{w'_0}$,
a contradiction. Consequently, $\{\mathbbm{1}_{\mathtt{j}},\theta\}$ is both a left,
a right and a two-sided cell.

Consider the matrix $[\theta]$ in the cell $2$-representation $\mathbf{C}$ of $\cA^\checkmark$. 
Let $P_{\mathbbm{1}_{\mathtt{j}}}$ and $P_{\theta}$ be the indecomposable projective
in $\overline{\mathbf{C}}(\mathtt{j})$ corresponding to $\mathbbm{1}_{\mathtt{j}}$ and $\theta$,
respectively. Let $L_{\mathbbm{1}_{\mathtt{j}}}$ and $L_{\theta}$ be 
their respective simple tops.  With respect to the basis
$P_{\mathbbm{1}_{\mathtt{j}}}$, $P_{\theta}$, the matrix $[\theta]$ has the form
\begin{displaymath}
\left(\begin{array}{cc}0&b\\1&a\end{array}\right),
\end{displaymath}
for some non-negative integers $a$ and $b$. As cell $2$-representation is simple transitive,
we must have $b>0$. As $\mathbbm{1}_{\mathtt{j}}$ acts on $\mathbf{C}(\mathtt{j})$ as the identity functor,
from \cite[Theorem~2]{KMMZ} it follows that $\overline{\mathbf{C}}(\mathtt{j})$ is
semi-simple. We also have the matrix bookkeeping the composition multiplicities 
in $\theta\, L_{\mathbbm{1}_{\mathtt{j}}}$ and $\theta\, L_{\theta}$.
By \cite[Lemma~8]{AM}, this matrix is
\begin{displaymath}
\left(\begin{array}{cc}0&1\\b&a\end{array}\right).
\end{displaymath}
We know that, on the one hand, $\theta\, P_{\mathbbm{1}_{\mathtt{j}}}\cong P_{\theta}$ but, 
on the other hand $L_{\theta}$ has composition multiplicity $b$ in $\theta\, L_{\mathbbm{1}_{\mathtt{j}}}$.
As $\overline{\mathbf{C}}(\mathtt{j})$ is semisimple, we have $P_{\mathbbm{1}_{\mathtt{j}}}\cong
L_{\mathbbm{1}_{\mathtt{j}}}$ and $P_{\theta}\cong L_{\theta}$. Therefore $b=1$.

Let $A:=\mathbb{C}[x]/(x^2-ax-1)$. Then $A$ is positively based 
with basis $1,x$ (see Subsection~\ref{s2.51}). It has a unique two-sided cell and the special module for this two-sided
cell is the one-dimensional $A$-module on which $x$ acts via $\frac{a+\sqrt{a^2+4}}{2}$.
Note that $\frac{a+\sqrt{a^2+4}}{2}$ is rational (in case of a non-negative integer $a$) 
if and only if $a=0$. Therefore, in case $A$ has a one-dimensional module on which 
$x$ acts as an integer, we must have $a=0$. Consequently, we have 
\begin{displaymath}
[\theta]=\left(\begin{array}{cc}0&1\\1&0\end{array}\right).
\end{displaymath}
In case $a=0$, the algebra $A$ has, up to equivalence, exactly two based modules in which $1+x$
acts as a matrix with positive integral coefficients. These are the modules in which the matrix of
$x$ is one of the following two matrices:
\begin{equation}\label{eqeqnn5}
\left(\begin{array}{cc}0&1\\1&0\end{array}\right)\quad\text{ or }\quad(1). 
\end{equation}
Now we can use the usual arguments (like in \cite[Subsection~5.10]{MaMa}) to show that the first matrix
in \eqref{eqeqnn5} corresponds to the cell $2$-representation of $\cA^\checkmark$. Indeed, 
let $\mathbf{K}$ be a simple transitive $2$-representation of $\cA^\checkmark$ in which
$[\theta]$ is the first matrix in \eqref{eqeqnn5}. Sending $\mathbbm{1}_{\mathtt{j}}$ to a simple module 
$L$ in $\overline{\mathbf{K}}(\mathtt{j})$ induces a homomorphism from  the cell $2$-representation of $\cA^\checkmark$
to $\mathbf{K}$ which is an equivalence as $\overline{\mathbf{K}}(\mathtt{j})$ is semisimple
and $L=\mathbbm{1}_{\mathtt{j}}(L)$ and $\theta(L)$ are indecomposable projectives which generate
$\overline{\mathbf{K}}(\mathtt{j})$. This shows that the cell $2$-representation of $\cA^\checkmark$
is the only simple transitive $2$-representations of $\cA^\checkmark$ for which $[\theta]$
is the first matrix in \eqref{eqeqnn5}.

Clearly, the second matrix in \eqref{eqeqnn5} corresponds to a 
unique simple transitive $2$-rep\-re\-sen\-ta\-ti\-on of $\cA^\checkmark$ (where both $\mathbbm{1}_{\mathtt{j}}$
and $\theta$ act as the identity functors and only scalar natural transformations survive). 
Therefore we obtain that $\cA^\checkmark$ has at most two equivalence classes of simple transitive 
$2$-representations, as required. The claim of the theorem follows.
\end{proof}

\subsection{Large quotients of Soergel bimodules}\label{s4.3}

The $2$-category $\cS$ has a unique maximal two-sided cell, namely $\{\theta_{w_0}\}$,
where $w_0$ is the longest element of $W$. Furthermore, if we remove this two-sided cell,
in what remains there is again a unique maximal two-sided cell. In this subsection we
denote this two-sided cell by $\mathcal{J}'$. Under the map $w\mapsto ww_0$, which is a bijection
reversing the two-sided order, the cell $\mathcal{J}'$ corresponds to the two-sided
cell $\mathcal{J}$ of $W$ containing all simple reflections, see \cite[Subsection~4.1]{KMMZ}.
The minimal quotient of $\cS$ in which elements of $\mathcal{J}'$ survive will be called
the {\em large quotient} of $\cS$ and denoted $\cS^{\checkmark}$. Our terminology is
motivated by the contrast to the terminology in \cite{KMMZ}.
From \cite[Subsection~4.1]{KMMZ} it follows that, in case $(W,S)$ is simply laced, the 
unique maximal two-sided cell of $\cS^{\checkmark}$ is strongly regular. Therefore, in the
simply laced case,
every faithful simple transitive $2$-representation of $\cS^{\checkmark}$ is equivalent
to a cell $2$-representation by \cite[Theorem~18]{MM5}.

\begin{corollary}\label{cor57}
Assume that $W$ is of type $B_n$, for $n\geq 2$. Then every faithful simple transitive
$2$-representation of $\cS^{\checkmark}$ is equivalent to a cell $2$-representation. Moreover,
there are exactly  two equivalence classes of such $2$-representations.
\end{corollary}

\begin{proof}
We claim that $\mathcal{J}'$ is nice. By \cite[Remark~3.3]{KaLu}
or \cite[Corollary~6.2.10]{BB}, the map
$w\mapsto ww_0$ sends left cells to left cells, right cells to right cells, and
two-sided cells to two-sided cells. Hence the first two conditions from the definition
of a nice two-sided cell for $\mathcal{J}'$ follow from the corresponding conditions for $\mathcal{J}$,
cf. Subsection~\ref{s4.1}. It remains to prove that $\mathcal{J}'$ contains the longest element
in some parabolic subgroup. 

Assume that the Coxeter diagram of $W$ is 
\begin{displaymath}
\xymatrix{1\ar@{-}[r]^4&2\ar@{-}[r]&3\ar@{-}[r]&\dots\ar@{-}[r]&n}. 
\end{displaymath}
We identify simple reflections with the vertices of this diagram. Let $W'$ be the parabolic subgroup
generated by the simple reflections $1$, $2$,\dots, $n-1$, and $w'_0$ be the longest element in $W'$.
Then 
\begin{gather*}
w_0= (n(n-1)(n-2)\dots 212\dots (n-1)n)\dots(32123)(212)(1), \\
w'_0= ((n-1)(n-2)(n-3)\dots 212\dots (n-2)(n-1))\dots(32123)(212)(1),\\
w_0=(n(n-1)(n-2)\dots 212\dots (n-1)n)w'_0.
\end{gather*}
As the element $n(n-1)(n-2)\dots 212\dots (n-1)n$ belongs to $\mathcal{J}$, see
\cite[Subsection~4.2]{KMMZ}, it follows that $w'_0\in \mathcal{J}'$.

The above proves that $\mathcal{J}'$ is nice. Now the claim
follows directly from Theorem~\ref{thm55}.
\end{proof}

\begin{remark}\label{rem58}
{\em
The observations in this remark are based on the lists 
of two-sided cells in types $B_3$ and $B_4$ which were computed by Tobias Kildetoft 
using GAP and can be found in Section~\ref{sappendix} below.

If $W$ is of type $B_3$, then $\cS$ has only two two-sided cells which are not strongly regular,
namely, the two-sided cell $\mathcal{J}'$ defined above and the two-sided cell  $\mathcal{J}$ 
which contains all simple reflections. Simple transitive $2$-representations of $\cS$ with 
strongly regular apex are equivalent to  cell $2$-representations by \cite[Theorem~18]{MM5}. 
Simple transitive $2$-representations of $\cS$ with apex $\mathcal{J}$ 
are equivalent to cell $2$-representations by \cite[Theorem~37]{KMMZ}. 
Simple transitive $2$-representations of $\cS$
with apex $\mathcal{J}'$ are equivalent to cell $2$-representations by Corollary~\ref{cor57}.
This completes the classification of simple transitive $2$-representations of $\cS$ in type $B_3$.
In particular, all simple transitive $2$-representations of $\cS$ in type $B_3$ are equivalent
to cell $2$-representations.

Similarly, one can check that all non-strongly regular two-sided cells of type $B_4$ are nice.
In particular, it follows that all simple transitive $2$-representations of $\cS$ in type $B_4$ are equivalent
to cell $2$-representations.

In type $B_5$ there is a unique two-sided cell 
(for which the value of Lusztig's $\mathbf{a}$-function 
from \cite{Lu} is $11$) which is neither strongly regular nor nice. The classification of 
simple transitive $2$-representations of $\cS$ in type $B_5$ with apex being this two-sided cell
seems to be the smallest open case for the moment.
} 
\end{remark}

\begin{remark}\label{rem59}
{\em
The maximal two-sided cell for $\cS^{\checkmark}$ in Coxeter types $F_4$, $H_3$ and  $H_4$ is not nice
and hence the above approach does not work in these cases. 
} 
\end{remark}

\section{Appendix: Cell structure in types $B_3$ and $B_4$}\label{sappendix}

This appendix contains lists of Kazhdan-Lusztig cells in types $B_3$ and $B_4$.
The lists were computed  by Tobias Kildetoft using the GAP package.
As we could not find any of these lists anywhere, we provide
them here. We thank Tobias Kildetoft for his help with these lists.

\subsection{Cells in type $B_3$}\label{sappendix.1}

We consider the Coxeter graph 
\begin{displaymath}
\xymatrix{1\ar@{-}[rr]^4&&2\ar@{-}[rr]&&3} 
\end{displaymath}
of type $B_3$ and the corresponding Coxeter group $(W,S)$, where $S=\{s_1,s_2,s_3\}$. 
Each element $w\in W$ will be given by one of its reduced expressions.
For simplicity, we denote the simple reflection $s_i$  by $i$.
For example, $1213$ denotes the element $s_1s_2s_1s_3$.

Boxes correspond to two-sided cells, given  with respect to the increasing
two-sided order (which is linear in this case). Rows are right cells and columns are left cells.
For each two-sided cell we also give the value of Lusztig's $\mathbf{a}$-function
on this cell, cf. \cite{Lu}. Elements given in bold are longest elements of parabolic subgroups.

The first two-sided cell. Its $\mathbf{a}$-value is $0$.
\begin{displaymath}
\begin{array}{|c|}
\hline
\mathbf{e}\\
\hline
\end{array} 
\end{displaymath}

The second two-sided cell. Its $\mathbf{a}$-value is $1$.
\begin{displaymath}
\begin{array}{|c|c|c|}
\hline
\mathbf{1}&12&123\\
121&&\\
\hline
21&\mathbf{2 }&23\\
&212&2123\\
\hline
321&32&\mathbf{3}\\
&3212&32123\\
\hline
\end{array} 
\end{displaymath}

The third two-sided cell. Its $\mathbf{a}$-value is $2$.
\begin{displaymath}
\begin{array}{|c|c|c|}
\hline
121321&1213&12132\\
\hline
1321&\mathbf{13}&132\\
\hline
21321&213&2132\\
\hline
\end{array} 
\end{displaymath}

The fourth two-sided cell. Its $\mathbf{a}$-value is $3$.
\begin{displaymath}
\begin{array}{|c|c|c|}
\hline
12321&1232&123212\\
\hline
2321&\mathbf{232}&23212\\
\hline
212321&21232&2123212\\
\hline
\end{array} 
\end{displaymath}

The fifth two-sided. Its $\mathbf{a}$-value is $4$.
\begin{displaymath}
\begin{array}{|c|c|c|}
\hline
\mathbf{1212}&12123&121232\\
12123212&1212321&\\
\hline
13212&132123&1232123\\
1213212&12132123&\\
\hline
213212&2132123&232123\\
&&21232123\\
\hline
\end{array} 
\end{displaymath}

The sixth two-sided. Its $\mathbf{a}$-value is $9$.
\begin{displaymath}
\begin{array}{|c|}
\hline
\mathbf{121232123}\\
\hline
\end{array} 
\end{displaymath}

\subsection{Cells in type $B_4$}\label{sappendix.2}

Here we use the same conventions as in Subsection~\ref{sappendix.1}
and consider the Coxeter graph 
\begin{displaymath}
\xymatrix{1\ar@{-}[rr]^4&&2\ar@{-}[rr]&&3\ar@{-}[rr]&&4} 
\end{displaymath}
Two-sided cells are again listed  with respect to the increasing
two-sided order. This order is linear with one exception that the two two-sided
cells with the same $\mathbf{a}$-value are not comparable.

The first two-sided cell. Its $\mathbf{a}$-value is $0$.
\begin{displaymath}
\begin{array}{|c|}
\hline
\mathbf{e}\\
\hline
\end{array} 
\end{displaymath}

The second two-sided cell. Its $\mathbf{a}$-value is $1$.
\begin{displaymath}
\begin{array}{|c|c|c|c|}
\hline
\mathbf{1}&12&123&1234\\
121&&&\\
\hline
21&\mathbf{2 }&23&234\\
&212&2123&21234\\
\hline
321&32&\mathbf{3}&34\\
&3212&32123&321234\\
\hline
4321&432&43&\mathbf{4}\\
&43212&432123&4321234\\
\hline
\end{array} 
\end{displaymath}

The third two-sided cell. Its $\mathbf{a}$-value is $2$.
\begin{displaymath}
\begin{array}{|c|c|c|c|c|c|}
\hline
3243&324&321432&3214&32143&3214321\\
321243&32124&&&&\\
\hline
243&\mathbf{24}&21432&214&2143&214321\\
21243&2124&&&&\\
\hline
213243&21324&2132&2134&213&21321\\
&&21321432&213214&2132143&213214321\\
\hline
1243&124&1432&\mathbf{14}&143&14321\\
&&121432&1214&12143&1214321\\
\hline
13243&1324&132&134&\mathbf{13}&1321\\
&&1321432&13214&132143&13214321\\
\hline
1213243&121324&12132&12134&1213&121321\\
&&121321432&1213214&12132143&1213214321\\
\hline
\end{array} 
\end{displaymath}

The fourth two-sided cell. Its $\mathbf{a}$-value is $3$.
{\tiny
\begin{displaymath}
\begin{array}{|c|c|c|c|c|c|c|c|}
\hline
32123432123&3212343&321243212&3212343212&32123432&3212432&321234321&32124321\\
\hline
3432123&\mathbf{343}&3243212&343212&3432&32432&34321&324321\\
\hline
212321243&2123243&2123212&21232124&212324&21232&2123214&212321\\
\hline
2123432123&212343&21243212&212343212&2123432&212432&21234321&2124321\\
\hline
23432123&2343&243212&2343212&23432&2432&234321&24321\\
\hline
2321243&23243&23212&232124&2324&\mathbf{232}&23214&2321\\
\hline
123432123&12343&1243212&12343212&123432&12432&1234321&124321\\
\hline
12321243&123243&123212&1232124&12324&1232&123214&12321\\
\hline
\end{array} 
\end{displaymath}
}

The fifth two-sided cell. Its $\mathbf{a}$-value is $4$.
\begin{displaymath}
\begin{array}{|c|c|c|c|c|c|}
\hline
212321243212&2123212432&212321432&21232143&2123214321&21232124321\\
\hline
2321243212&23212432&2321432&232143&23214321&232124321\\
\hline
213243212&2132432&213432&21343&2134321&21324321\\
\hline
13243212&132432&13432&\mathbf{1343}&134321&1324321\\
\hline
1213243212&12132432&1213432&121343&12134321&121324321\\
\hline
12321243212&123212432&12321432&1232143&123214321&1232124321\\
\hline
\end{array} 
\end{displaymath}

The sixth two-sided cell. Its $\mathbf{a}$-value is $4$.
\begin{displaymath}
\begin{array}{|c|c|c|c|c|c|}
\hline
321234321234&32124321234&3212432123&3214321234&321432123&32143212\\
34321234&324321234&32432123&&&\\
\hline
21234321234&2124321234&212432123&214321234&21432123&2143212\\
234321234&24321234&2432123&&&\\
\hline
2123212343&212321234&21232123&21321234&2132123&213212\\
23212343&2321234&232123&&&\\
\hline
1234321234&124321234&12432123&14321234&1432123&143212\\
&&&1214321234&121432123&12143212\\
\hline
123212343&12321234&1232123&1321234&132123&13212\\
&&&121321234&12132123&1213212\\
\hline
12123243&1212324&121232&121234&12123&\mathbf{1212}\\
&&&12123214&1212321&12123212\\
\hline
\end{array} 
\end{displaymath}

The seventh two-sided cell. Its $\mathbf{a}$-value is $5$.
{\tiny
\begin{displaymath}
\begin{array}{|c|c|c|c|c|c|c|c|}
\hline
21324321234&2132432123&2134321234&213432123&21343212&213214321234&21321432123&2132143212\\
\hline
2132123432&213212432&213212343&21321243&2132124&21321234321&2132124321&21321243212\\
\hline
1324321234&132432123&134321234&13432123&1343212&13214321234&1321432123&132143212\\
\hline
132123432&13212432&13212343&1321243&132124&1321234321&132124321&1321243212\\
\hline
12123432&1212432&1212343&121243&\mathbf{12124}&121234321&12124321&121243212\\
\hline
121324321234&12132432123&12134321234&1213432123&121343212&1213214321234&121321432123&12132143212\\
\hline
12132123432&1213212432&1213212343&121321243&12132124&121321234321&12132124321&121321243212\\
\hline
1212321432&12123212432&121232143&1212321243&121232124&12123214321&121232124321&1212321243212\\
\hline
\end{array} 
\end{displaymath}
}

The eights two-sided cell. Its $\mathbf{a}$-value is $6$.
{\tiny
\begin{displaymath}
\begin{array}{|c|c|c|c|c|c|}
\hline
21232123432123&2123212343212&21232123432&2132123432123&213212343212&212321234321\\
232123432123&23212343212&232123432&&&2321234321\\
\hline
2123212432123&21232124321234&212324321234&212321432123&21232143212&2123214321234\\
21232432123&2123243212&21232432&&&212324321\\
\hline
232432123&232124321234&2324321234&2321432123&232143212&23214321234\\
23212432123&23243212&\mathbf{232432}&&&2324321\\
\hline
1232123432123&123212343212&1232123432&132123432123&13212343212&12321234321\\
&&&12132123432123&1213212343212&\\
\hline
121232432123&12123243212&121232432&12123432123&1212343212&1212324321\\
&&&1212321432123&121232143212&\\
\hline
123212432123&1232124321234&12324321234&12321432123&1232143212&123214321234\\
1232432123&123243212&1232432&&&12324321\\
\hline
\end{array} 
\end{displaymath}
}

The ninth two-sided cell. Its $\mathbf{a}$-value is $9$.
\begin{displaymath}
\begin{array}{|c|c|c|c|}
\hline
212321234321234&21321234321234&2132124321234&213212432123\\
2321234321234&&&\\
\hline
12321234321234&1321234321234&132124321234&13212432123\\
&121321234321234&12132124321234&1213212432123\\
\hline
1212324321234&121234321234&12124321234&1212432123\\
&12123214321234&121232124321234&12123212432123\\
\hline
121232123432&12123212343&1212321234&\mathbf{121232123}\\
&1212321234321&12123212343212&121232123432123\\
\hline
\end{array} 
\end{displaymath}

The tenth two-sided cell. Its $\mathbf{a}$-value is $16$.
\begin{displaymath}
\begin{array}{|c|}
\hline
\mathbf{1212321234321234}\\
\hline
\end{array} 
\end{displaymath}

\vspace{5mm}

\noindent
M.~M.: Center for Mathematical Analysis, Geometry, and Dynamical Systems, Departamento de Matem{\'a}tica, 
Instituto Superior T{\'e}cnico, 1049-001 Lisboa, PORTUGAL \& Departamento de Matem{\'a}tica, FCT, 
Universidade do Algarve, Campus de Gambelas, 8005-139 Faro, PORTUGAL, email: {\tt mmackaay\symbol{64}ualg.pt}

\noindent
Vo.~Ma.: Department of Mathematics, Uppsala University, Box. 480,
SE-75106, Uppsala, SWEDEN, email: {\tt mazor\symbol{64}math.uu.se}

\noindent
Va.~Mi.: School of Mathematics, University of East Anglia,
Norwich NR4 7TJ, UK,  {\tt v.miemietz\symbol{64}uea.ac.uk}

\noindent
X.~Z.: Department of Mathematics, Uppsala University, Box. 480,
SE-75106, Uppsala, SWEDEN, email: {\tt xiaoting.zhang\symbol{64}math.uu.se}


\begin{thebibliography}{9999999}
\bibitem[Au]{Au} M.~Auslander. Representation theory of Artin algebras. I, II. 
Comm. Algebra {\bf 1} (1974), 177--268; ibid. {\bf 1} (1974), 269--310. 
\bibitem[AM]{AM} T.~Agerholm, V.~Mazorchuk. On selfadjoint functors satisfying 
polynomial relations. J. Algebra {\bf 330} (2011), 448--467. 
\bibitem[BFK]{BFK} J. Bernstein, I. Frenkel, M. Khovanov. A categorification of the Temperley-Lieb algebra
and Schur quotients of $U(\mathfrak{sl}_2)$ via projective and Zuckerman functors. Selecta Math.
(N.S.) \textbf{5} (1999), no. 2, 199--241.
\bibitem[BB]{BB} A.~Bj{\"o}rner, F.~Brenti. Combinatorics of Coxeter groups. 
Graduate Texts in Mathematics {\bf 231}. Springer, New York, 2005.
\bibitem[CM]{CM} A.~Chan, V.~Mazorchuk. Diagrams and discrete extensions for finitary $2$-representations.
Preprint arXiv:1601.00080, to appear in Math. Proc. Cambr. Phil. Soc.
\bibitem[CR]{CR} J. Chuang, R. Rouquier. Derived equivalences for symmetric groups and
$\mathfrak{sl}_2$-categorification. Ann. of Math. (2) \textbf{167} (2008), no. 1, 245--298.
\bibitem[Cl]{Cl} A.~Clifford. Matrix representations of completely simple semigroups. 
Amer. J. Math. {\bf 64} (1942), 327--342. 
\bibitem[EW]{EW} B.~Elias, G.~Williamson. The Hodge theory of Soergel bimodules. 
Ann. of Math. (2) {\bf 180} (2014), no. 3, 1089--1136.
\bibitem[EGNO]{EGNO} P. Etingof, S. Gelaki, D. Nikshych, V. Ostrik. Tensor categories. 
Mathematical Surveys and Monographs {\bf 205}. American Mathematical Society, Providence, RI, 2015.
\bibitem[ENO]{ENO} P. Etingof, D. Nikshych, V. Ostrik. On fusion categories. Ann. of 
Math. (2) {\bf 162} (2005), no. 2, 581--642.
\bibitem[EO]{EO} P. Etingof, V. Ostrik. Finite tensor categories. Mosc. Math. J. 
{\bf 4} (2004), no. 3, 627--654, 782--783. 
\bibitem[GMS]{GMS} O.~Ganyushkin, V.~Mazorchuk, B.~Steinberg. On the irreducible 
representations of a finite semigroup. Proc. Amer. Math. Soc. {\bf 137} (2009), no. 11, 
3585--3592.
\bibitem[GL]{GL} R.~Green, J.~Losonczy. Fully commutative Kazhdan-Lusztig cells. 
Ann. Inst. Fourier (Grenoble) {\bf 51} (2001), no. 4, 1025--1045.
\bibitem[GM1]{GM1} A. L. Grensing, V. Mazorchuk. Categorification of the Catalan monoid.
Semigroup Forum \textbf{89} (2014), no. 1, 155--168.
\bibitem[GM2]{GM2} A. L. Grensing, V. Mazorchuk. Categorification using dual projection functors.
Commun. Contemp. Math. {\bf 19} (2017), no. 3, 1650016, 40 pp.
\bibitem[KaLu]{KaLu} D.~Kazhdan, G.~Lusztig. Representations of Coxeter groups and Hecke 
algebras. Invent. Math. {\bf 53} (1979), no. 2, 165--184.
\bibitem[Ke]{Ke} G..~Kelly. Doctrinal adjunction. Category Seminar (Proc. Sem., Sydney, 
1972/1973), pp. 257--280. Lecture Notes in Math., Vol. {\bf 420}, Springer, Berlin, 1974.
\bibitem[KhM]{KhM}  O.~Khomenko, V.~Mazorchuk. On Arkhipov's and Enright's functors.
Math. Z. {\bf 249} (2005), no. 2, 357--386. 
\bibitem[KL]{KL} M.~Khovanov, A.~Lauda. A categorification of a quantum $\mathfrak{sl}_n$. 
Quantum Topol. {\bf 1} (2010), 1--92.
\bibitem[KiM]{KM} T. Kildetoft, V. Mazorchuk. Special modules over positively based 
algebras. Documenta Math. {\bf 21} (2016) 1171--1192.
\bibitem[KMMZ]{KMMZ} T.~Kildetoft, M.~Mackaay, V.~Mazorchuk, J.~Zimmermann. Simple transitive 
$2$-rep\-re\-sen\-ta\-ti\-ons of small quotients of Soergel bimodules. Preprint arXiv:1605.01373.
To appear in Trans. Amer. Math. Soc.
\bibitem[Le]{Le} T. Leinster. Basic bicategories, Preprint arXiv:math/9810017.
\bibitem[Lu]{Lu} G.~Lusztig. Cells in affine Weyl groups. Algebraic groups and related topics 
(Kyoto/Nagoya, 1983), 255--287, Adv. Stud. Pure Math., {\bf 6}, North-Holland, Amsterdam, 1985. 
\bibitem[McL]{McL} S. Mac Lane. Categories for the Working Mathematician. Springer-Verlag, 1998.
\bibitem[MaMa]{MaMa} M.~Mackaay, V.~Mazorchuk. Simple transitive $2$-representations for some 
$2$-subcategories of Soergel bimodules. J. Pure Appl. Algebra  {\bf 221} (2017), no. 3, 565--587.
\bibitem[MT]{MT} M.~Mackaay, D.~Tubbenhauer. Two-color Soergel calculus and simple transitive 
$2$-rep\-re\-sen\-tations. Preprint arXiv:1609.00962. 
\bibitem[MMMT]{MMMT} M.~Mackaay, V.~Mazorchuk, V.~Miemietz and D.~Tubbenhauer. 
Simple transitive $2$-representations via (co)algebra $1$-morphism. Preprint arXiv:1612.06325.
To appear in Indiana Univ. Math. J.
\bibitem[Ma1]{Ma1} V. Mazorchuk. Lectures on algebraic categorification. QGM Master Class Series.
European Mathematical Society (EMS), Zurich, 2012.
\bibitem[Ma2]{Ma2} V. Mazorchuk. Classification problems in $2$-representation theory.
S\~ao Paulo J. Math. Sci. 11 (2017), no. 1, 1--22. 
\bibitem[MM1]{MM1} V. Mazorchuk, V. Miemietz. Cell 2-representations of finitary 2-categories.
Compositio Math \textbf{147}  (2011), 1519--1545.
\bibitem[MM2]{MM2} V. Mazorchuk, V. Miemietz. Additive versus abelian 2-representations of fiat 2-categories.
Moscow Math. J. \textbf{14} (2014), No.3, 595--615.
\bibitem[MM3]{MM3} V. Mazorchuk, V. Miemietz. Endmorphisms of cell 2-representations. 
Int. Math. Res. Notes, Vol. {\bf 2016}, No. 24, 7471--7498.
\bibitem[MM4]{MM4} V. Mazorchuk, V. Miemietz. Morita theory for finitary 2-categories.
Quantum Topol. \textbf{7} (2016), no. 1, 1--28. 
\bibitem[MM5]{MM5} V. Mazorchuk, V. Miemietz. Transitive 2-representations of finitary 2-categories.
Trans. Amer. Math. Soc. {\bf 368} (2016), no. 11, 7623--7644.
\bibitem[MM6]{MM6} V. Mazorchuk, V. Miemietz. Isotypic faithful 2-representations of $\mathcal{J}$-simple
fiat 2-ca\-te\-go\-ri\-es. Math. Z. \textbf{282} (2016), no. 1-2, 411--434.
\bibitem[MMZ1]{MMZ1} V. Mazorchuk, V. Miemietz, X.~Zhang. Characterisation and applications of 
$\Bbbk$-split bimodules. Preprint arXiv:1701.03025. 
\bibitem[MMZ2]{MMZ2} V. Mazorchuk, V. Miemietz, X.~Zhang. Pyramids and $2$-representations.
Preprint arXiv:1705.03174. 
\bibitem[MZ1]{MZ1} V.~Mazorchuk, X.~Zhang. Simple transitive $2$-representations for two non-fiat 
$2$-categories of projective functors. Preprint arXiv:1601.00097. To appear in Ukr. Math. J.
\bibitem[MZ2]{MZ2} V.~Mazorchuk, X.~Zhang. Bimodules over uniformly oriented $A_n$ 
quivers with radical square zero. Preprint arXiv:1703.08377.
\bibitem[Mu]{Mu} W.~Munn. Matrix representations of semigroups. Proc. 
Cambridge Philos. Soc. {\bf 53} (1951), 5--12. 
\bibitem[Os]{Os} V. Ostrik. Module categories, weak Hopf algebras and modular invariants. 
Transform. Groups {\bf 8} (2003), no. 2, 177--206.
\bibitem[Po]{Po} I.~Ponizovski{\"\i}. On matrix representations of associative systems. 
Mat. Sb. N.S., {\bf 38} (1956), 241--260. 
\bibitem[Ro]{Ro} R. Rouquier. 2-Kac-Moody algebras. Preprint arXiv:0812.5023.
\bibitem[So1]{So1} W.~Soergel. The combinatorics of Harish-Chandra bimodules. 
J. Reine Angew. Math. {\bf 429} (1992), 49--74. 
\bibitem[So2]{So2} W.~Soergel. Kazhdan-Lusztig-Polynome und unzerlegbare Bimoduln {\"u}ber Polynomringen. 
J. Inst. Math. Jussieu {\bf 6} (2007), no. 3, 501--525.
\bibitem[St]{St} J.~Stembridge. On the fully commutative elements of Coxeter groups. 
J. Algebraic Combin. {\bf 5} (1996), no. 4, 353--385.
\bibitem[Ta]{Ta} M. Takeuchi. Morita theorems for categories of comodules. J. Fac. Sci. Univ. Tokyo Sect. IA Math. {\bf 24} (1977), no. 3, 629--644.
\bibitem[Xa]{Xa} Q.~Xantcha. Gabriel $2$-quivers for finitary $2$-categories. J. Lond. Math. Soc. 
(2) {\bf 92} (2015), no. 3, 615--632.
\bibitem[Zh1]{Zh1} X.~Zhang. Duflo involutions for $2$-categories associated to tree quivers. 
J. Algebra Appl. {\bf 15} (2016), no. 3, 1650041, 25 pp.
\bibitem[Zh2]{Zh2} X.~Zhang. Simple transitive $2$-representations and Drinfeld center 
for some finitary $2$-categories. J. Pure Appl. Alg. {\bf 222} (2018), 97--130.
\bibitem[Zi1]{Zi1} J.~Zimmermann. Simple transitive $2$-representations of Soergel bimodules 
in type $B_2$. J. Pure Appl. Algebra {\bf 221} (2017), no. 3, 666--690. 
\bibitem[Zi2]{Zi2} J.~Zimmermann. Simple transitive 2-representations of some 2-categories of projective functors.
Preprint  arXiv:1705.01149. To appear in Beitr. Alg. Geom.
\end{thebibliography}
\end{document}